\documentclass[12pt,a4paper,reqno]{article}
\pdfoutput=1
\usepackage{amsmath,amsfonts}
\usepackage{amsthm}
\usepackage{amscd}
\usepackage{enumerate}
\usepackage{authblk}
\usepackage[margin=1in]{geometry}
\usepackage{tikz}
\usepackage[colorlinks]{hyperref}

\title{Eventually Positive Semigroups of Linear Operators}

\author[1]{Daniel Daners}%
\author[2]{Jochen Gl\"uck\thanks{Supported by a scholarship within the
    scope of the LGFG Baden-W\"urttemberg, Germany.}}%
\author[3]{James B. Kennedy\thanks{Partly supported by a fellowship of
    the Alexander von Humboldt Foundation, Germany.}}%
\affil[1]{School of Mathematics and Statistics, University of Sydney,
  NSW 2006, Australia\authorcr%
  \nolinkurl{daniel.daners@sydney.edu.au}}%
\affil[2]{Institut f\"ur Angewandte Analysis, Universit\"at Ulm,
  D-89069 Ulm, Germany\authorcr%
  \nolinkurl{jochen.glueck@uni-ulm.de}}%
\affil[3]{Institut f\"ur Analysis, Dynamik und Modellierung,
  Universit\"at Stuttgart, Pfaffenwaldring~57, D-70659 Stuttgart,
  Germany\authorcr%
  \nolinkurl{james.kennedy@mathematik.uni-stuttgart.de}}%

\date{August 8, 2015}

\hypersetup{%
  pdfauthor={Daniel Daners, Jochen Glueck, James Kennedy},%
  pdftitle={Eventually positive semigroups of linear operators},%
  pdfkeywords={One-parameter semigroups of linear operators, semigroups
    on Banach lattices, eventually positive semigroup, Perron-Frobenius
    theory}%
}

\theoremstyle{plain}
\newtheorem{theorem}{Theorem}[section]
\newtheorem{lemma}[theorem]{Lemma}
\newtheorem{proposition}[theorem]{Proposition}
\newtheorem{corollary}[theorem]{Corollary}

\theoremstyle{definition}
\newtheorem{definition}[theorem]{Definition}

\newtheorem{example}[theorem]{Example}

\theoremstyle{remark}
\newtheorem{remark}[theorem]{Remark}

\numberwithin{equation}{section}
\numberwithin{figure}{section}

\DeclareMathOperator{\dist}{dist}
\DeclareMathOperator{\one}{\mathbf 1}
\DeclareMathOperator{\im}{im}
\DeclareMathOperator{\spr}{r}
\DeclareMathOperator{\spb}{s}
\DeclareMathOperator{\gbd}{\omega_0}
\DeclareMathOperator{\repart}{Re}

\DeclareMathOperator{\per}{per}
\DeclareMathOperator{\abs}{abs}

\newcommand{\bbN}{\mathbb{N}}
\newcommand{\bbZ}{\mathbb{Z}}
\newcommand{\bbR}{\mathbb{R}}

\newcommand{\bbC}{\mathbb{C}}
\newcommand{\bbT}{\mathbb{T}}
\newcommand{\calL}{\mathcal{L}}
\newcommand{\calT}{\mathcal{T}}

\newcommand{\phdot}{\mathord{\,\cdot\,}}

\let\oldthebibliography\thebibliography
\renewcommand\thebibliography[1]{
  \oldthebibliography{#1}
  \setlength{\parskip}{0pt}
  \setlength{\itemsep}{0pt plus 0.3ex}
  \small
}

\begin{document}
\maketitle

\begin{abstract}
  We develop a systematic theory of eventually positive semigroups of
  linear operators mainly on spaces of continuous functions. By
  eventually positive we mean that for every positive initial condition
  the solution to the corresponding Cauchy problem is positive for large
  enough time. Characterisations of such semigroups are given by means
  of resolvent properties of the generator and Perron--Frobenius type
  spectral conditions. We apply these characterisations to prove
  eventual positivity of several examples of semigroups including some
  generated by fourth order elliptic operators and a delay differential
  equation. We also consider eventually positive semigroups on arbitrary
  Banach lattices and establish several results for their spectral bound
  which were previously only known for positive semigroups.
\end{abstract}

{ %Provide MSC, Keywords, Dedication
  \renewcommand{\thefootnote}{} \footnotetext{Dedicated with great
    pleasure to Wolfgang Arendt on the occasion of his
    65\textsuperscript{th} birthday}%
  \footnotetext{Published in J. Math. Anal. Appl. 433 (2016),
    1561--1593. DOI:
    \href{http://dx.doi.org/10.1016/j.jmaa.2015.08.050}{10.1016/j.jmaa.2015.08.050}}
}

\section{Introduction}
\label{sec:introduction}
One of the distinguishing features of many second-order parabolic
boundary value problems is their positivity preserving property: if the
initial condition is positive, so is the solution at all positive
times. Such equations are frequently expressed as an abstract Cauchy
problem of the form
\begin{equation}
  \label{eq:ivp}
  \dot u(t)=Au(t)\quad\text{if $t \geq 0$,}\qquad
  u(0)=u_0,
\end{equation}
on a (complex) Banach lattice $E$ such as $L^p(\Omega)$ or
$C(\overline\Omega)$, where $A$ is the generator of a strongly
continuous semigroup. If we represent the solution of \eqref{eq:ivp} in
terms of the corresponding semigroup $(e^{tA})_{t\geq 0}$, then
positivity means that $u_0\geq 0$ implies $e^{tA}u_0\geq 0$ for all
$t\geq 0$. There is a sophisticated general theory of positive
semigroups, which has found a large number of applications; see for
instance \cite{Arendt1986}.

However, if $A$ is the realisation of a \emph{differential} operator,
such positivity---which is usually obtained as a consequence of the
maximum principle---is surprisingly rare. Indeed, under mild auxiliary
assumptions on the operator $A$, \emph{a priori} of arbitrary order,
positivity of the semigroup $(e^{tA})_{t\geq 0}$ already implies that
$A$ is second-order elliptic if $E = L^p(\bbR^d)$ \cite{MiOk86} or
$E = C(\overline\Omega)$ \cite{ABR90}.

In such a case, in an attempt to bypass this restriction, we could
weaken the requirement on the semigroup and stipulate that $e^{tA}u_0$
merely be positive for $t\geq 0$ ``large enough'' whenever $u_0 \geq 0$.
Indeed, in recent times various disparate examples of such ``eventually
positive semigroups'' have emerged, all seemingly completely independent
of each other. Here are some examples, many of which we will consider in
more detail below, in Section~\ref{section_applications_on_c_k}.

A matrix exponential $e^{tA}$ can be positive for large $t$ even if $A$
has some negative off-diagonal entries, a phenomenon which seems to have
been observed only quite recently; see \cite{Noutsos2008} and the
references therein.

For elliptic operators of order $2m$, $m\geq 2$, there is no maximum
principle in general. The resolvent of the bi-Laplacian exhibits
positivity properties on very few domains such as balls and
perturbations of balls; see \cite{AcSw05,GrSw96}. The question as to
whether the corresponding parabolic problem becomes ``essentially''
positive for large $t>0$ was investigated in \cite{FGG08,GaGr08}.

Another example is the \emph{Dirichlet-to-Neumann operator} $D_\lambda$
associated with $\Delta u + \lambda u = 0$ on a domain $\Omega$. For
$\lambda$ on one side of the first Dirichlet eigenvalue, the semigroup
generated by $-D_\lambda$ on $L^2(\partial\Omega)$ is positive as shown
in \cite{arendt:12:fei}. For other values of $\lambda$ the semigroup may
be positive, eventually positive or neither as the example of the disc
shows; see \cite{Daners2014}. The present paper had its genesis in an
attempt to understand this phenomenon better at a theoretical level. We
provide a detailed discussion in
Section~\ref{section_applications_on_c_k}.

A further example is provided by certain delay differential
equations. Under special assumptions they generate positive semigroups;
see \cite[Theorem~VI.6.11]{Engel2000}. We will show in 
Section~\ref{subsection:appl-delay-equ}
that there are also situations where the semigroup is eventually 
positive without being positive

The variety of examples suggests that eventually positive semigroups
could prove more ubiquitous than their positive counterparts, and no
doubt more examples will emerge. Quite surprisingly, to date there 
seems to have been no unified treatment of such objects, in marked 
contrast to the positive case. Here, and in an envisaged sequel
\cite{DanersII}, we intend to address this. Our abstract theory will
allow us to recover several known results and to prove some new ones in
the above-mentioned areas.

Our main focus in this article is the investigation of strongly
continuous semigroups with eventual positivity properties on $C(K)$, the
space of complex-valued continuous functions on a compact non-empty
Hausdorff space $K$. In order to give an idea of our results, we first
need to introduce some notation. We call $f$ \emph{positive} if
$f(x)\geq 0$ for all $x\in K$ and write $f\geq 0$. If $f\geq 0$ but
$f\neq 0$ we write $f>0$; we call $f$ \emph{strongly positive} and write
$f\gg 0$ if there exists $\beta>0$ such that $f\geq\beta\one$, where
$\one$ is the constant function on $K$ with value one. A bounded linear
operator $T$ on $C(K)$ is called \emph{strongly positive}, denoted by $T
\gg 0$, if $Tf \gg 0$ whenever $f > 0$, and similarly, a linear
functional $\varphi: C(K) \to \bbC$ is called \emph{strongly positive},
again denoted by $\varphi \gg 0$, if $\varphi(f) > 0$ for each $f > 0$.

If $A\colon D(A)\to E$ is a closed operator with domain $D(A)$ on the
Banach space $E$ we denote by $\sigma(A)$ and
$\varrho(A):=\bbC\setminus\sigma(A)$ the \emph{spectrum} and
\emph{resolvent set} of $A$, respectively. Any point in $\sigma(A)$ is
called a \emph{spectral value}. We call
\begin{equation}
  \label{eq:spb}
  \spb(A):=\sup\bigl\{\repart\lambda\colon\lambda\in\sigma(A)\bigr\}
  \in [-\infty, \infty]
\end{equation}
the \emph{spectral bound} of $A$. For some classes of positive
semigroups the spectral bound $\spb(A)$ is necessarily a \emph{dominant
  spectral value}, by which we mean that $\spb(A) \in \sigma(A)$ and
that the \emph{peripheral spectrum}
\begin{equation}
  \label{eq:sigma-per}
  \sigma_{\per}(A):=\sigma(A)\cap(\spb(A)+i\bbR)
\end{equation}
consists of $\spb(A)$ only. As in the theory of positive semigroups
there is a close relationship between positivity properties of the
resolvent
\begin{equation}
  \label{eq:resolvent}
  \varrho(A)\to\mathcal L(E),\quad
  \lambda\mapsto R(\lambda,A):=(\lambda I -A)^{-1}
\end{equation}
and the semigroup $(e^{tA})_{t\geq 0}$; we shall also see that the
spectral projection $P$ associated with $\spb(A)$ plays an important
role here in the case that $\spb(A)$ is an isolated spectral value.  The
main part of this paper is, roughly speaking, devoted to characterising
the relationship between these three objects. The following theorem
provides a rather incomplete but indicative snapshot of our results.
\begin{theorem}
  \label{thm:snapshot}
  Let $A$ be the generator of a strongly continuous real semigroup
  $(e^{tA})_{t \ge 0}$ on $C(K)$ with spectral bound
  $\spb(A)>-\infty$. Assume that $\sigma_{\per}(A)$ is finite and
  consists of poles of the resolvent.  Then the following assertions are
  equivalent:
  \begin{enumerate}[\upshape (i)]
  \item For every $f>0$ there exists $t_0>0$ such that $e^{tA}f\gg 0$
    for all $t\geq t_0$;
  \item The semigroup $(e^{t(A-\spb(A) I)})_{t\geq 0}$ is bounded,
    $\spb(A)$ is a dominant spectral value and for every $f>0$ there
    exists $\lambda_0>\spb(A)$ such that $R(\lambda,A)f\gg 0$ for all
    $\lambda\in(\spb(A),\lambda_0]$;
  \item The semigroup $(e^{t(A-\spb(A) I)})_{t\geq 0}$ is bounded,
    $\spb(A)$ is a dominant spectral value and a geometrically simple
    eigenvalue; moreover, $\ker(\spb(A) I - A)$ and $\ker(\spb(A) I -
    A')$ each contain a strongly positive vector.
  \item The semigroup $(e^{t(A-\spb(A) I)})_{t\geq 0}$ is bounded,
    $\spb(A)$ is a dominant spectral value and the associated spectral
    projection $P$ fulfils $P\gg 0$.
  \end{enumerate}
\end{theorem}

Let us briefly comment on the role of the different assertions in the
theorem: The implication ``(i) $\Rightarrow$ (iii)'' is a typical
Perron--Frobenius type result, since it infers the existence of a
dominant spectral value and corresponding positive eigenvectors from
positivity properties of the semigroup.  The converse implication
``(iii) $\Rightarrow$ (i)'' has no analogue for positive semigroups and
is characteristic for \emph{eventual} positivity.  The close
relationship between properties of the semigroup in (i) and the
resolvent in (ii) is reminiscent of the theory of positive semigroups,
where the semigroup is positive if and only if the resolvent is positive
for all sufficiently large $\lambda$; see
\cite[Theorem~VI.1.8]{Engel2000}. Here, however, it turns out that we
need to consider small $\lambda$ and an additional spectral condition on
the generator $A$. To consider the spectral projection as in (iv) does
not seem to be common in classical Perron--Frobenius theory; however, we
shall see that it is essential for relating the other conditions to each
other.

It is our intention not just to prove a blanket result like
Theorem~\ref{thm:snapshot}, but to give a more detailed analysis of each
of the different objects considered in the theorem, namely the spectral
projection, the resolvent and the semigroup.  This is done in
Sections~\ref{sec:projections}--\ref{sec:semigroups}.

Section~\ref{sec:projections} is concerned with the characterisation of
eigenvalues $\lambda_0$ for which $A$ and its dual $A'$ have a strongly
positive eigenvector and for which the corresponding eigenspace of $A$
is one-dimensional. We show in Proposition~\ref{prop:projections} that
this is equivalent to the spectral projection $P$ associated with
$\lambda_0$ being strongly positive.

In Section~\ref{sec:resolvents} we characterise strongly positive
projections by means of eventual positivity properties of the resolvent.
Our main result on resolvents is given in
Theorem~\ref{thm:evpos-resolvents-c_k}.

Under some additional assumptions on the peripheral spectrum of $A$ we
characterise individually eventually strongly positive semigroups in
terms of resolvent and spectral projection in
Section~\ref{sec:semigroups}; see in particular
Theorem~\ref{thm:evpos-semigroups-c_k} and
Corollary~\ref{cor:evpos-semigroups-c_k}.

The assertions of Theorem~\ref{thm:snapshot} follow from combining
Proposition~\ref{prop:projections},
Theorem~\ref{thm:evpos-resolvents-c_k} and
Theorem~\ref{thm:evpos-semigroups-c_k}. Note that in general $t_0$ and
$\lambda_0$ in (i) and (ii) of Theorem~\ref{thm:snapshot} depend on the
choice of $f>0$: these are \emph{individual} rather than \emph{uniform}
conditions with respect to $f$. We will also illustrate the distinction
between individual and uniform eventual positivity in
Section~\ref{sec:semigroups}; see in particular
Examples~\ref{ex_ind_evtl_pos_does_not_imply_unif_evtl_pos}
and~\ref{ex_ind_evtl_pos_does_not_imply_unif_evtl_pos_for_compact_res}.

In Section~\ref{section_applications_on_c_k} we apply our theory to a
diverse range of examples: matrix exponentials, a Dirichlet-to-Neumann
operator, the square of the Laplacian with Robin boundary conditions,
and a delay differential equation. Here, the advantages of our approach
become apparent: in practice it is usually easy to check the required
condition on the spectral projection of the dominant eigenvalue, and
this characterises eventual positivity.

In Section~\ref{sec:spectral_bound} we prove various spectral properties
of operators generating eventually positive semigroups on arbitrary
Banach lattices, in particular regarding the spectral bound. These
generalise corresponding results about positive semigroups, and will be
needed in several other places throughout the
article. Section~\ref{sec:resolvents_final_remarks} is concerned with
some remarks on resolvents. The abstract results in
Sections~\ref{sec:spectral_bound} and~\ref{sec:resolvents_final_remarks}
stand apart from the main thrust of the paper, which is concerned with
characterisation theorems suitable for concrete applications. Therefore,
we defer them until the end.

In \cite{DanersII}, we will prove many similar results both for the
technically more complicated case of arbitrary Banach lattices and for
weaker forms of eventual positivity, which will allow us to establish
characterisations which are uniform in the function $f>0$. This can be
used to study various types of higher-order elliptic operators.

\section{Preliminaries}
\label{sec:prelims}
We briefly introduce some further notation and basic facts we use
throughout the paper. Whenever $E$ is a complex Banach space, we denote
by $\calL(E)$ the space of bounded linear operators on $E$. If $M
\subseteq E$, then $\langle M \rangle$ denotes the span of $M$ in $E$
(over $\bbC$). If $A \in \calL(E)$, then
\begin{equation}
  \label{eq:spr}
  \spr(A) := \sup\{|\lambda|\colon \lambda \in \sigma(A)\}
\end{equation}
denotes the \emph{spectral radius} of $A$. We will make extensive use of
properties of spectral projections, in particular in connection with
poles of the resolvent $R(\phdot,A)$. As an essential tool we make use
of the Laurent expansion of $R(\phdot,A)$ about isolated points of the
spectrum of $A$, which may be summarised as follows.
\begin{remark}
  \label{rem:Laurent-exp}
  Let $A$ be a closed operator on a Banach space $E$. The resolvent
  $\lambda\mapsto R(\lambda,A)$ is an analytic map on
  $\varrho(A)\subseteq\bbC\to\calL(E)$. If $\lambda_0$ is an isolated
  point of $\sigma(A)$, then there exist operators $U,P,B\in\calL(E)$
  such that the \emph{Laurent expansion}
  \begin{equation}
    \label{eq:Laurent-exp}
    R(\lambda,A)
    =\frac{P}{\lambda-\lambda_0}
    +\sum_{k=1}^\infty\frac{U^k}{(\lambda-\lambda_0)^{k+1}}
    +\sum_{k=0}^\infty(-1)^k(\lambda-\lambda_0)^kB^{k+1}
  \end{equation}
  is valid for $\lambda$ is some neighbourhood of $\lambda_0$.
  Moreover, $P$ is the spectral projection associated with $\lambda_0$,
  $U=-(\lambda_0 I-A)P$, $BP=PB=0$ and $UB=BU=0$; see
  \cite[Section~III-6.5]{Kato1976}, \cite[Section~VIII.8]{Yosida1995},
  \cite[p.\,246--248]{Engel2000} or \cite{campbell:13:lac}. The operator
  $U$ is called the (quasi-) nilpotent part of $A$ associated with
  $\lambda_0$.  For convenience we set $U^0:=P$. Then $U^n =
  (-1)^n(\lambda_0 I-A)^nP$ for all $n \in \bbN_0$.

  Assume for the rest of this remark that $\lambda_0$ is a \emph{pole of
    the resolvent}, that is, there exists a minimal number $m\geq 1$,
  the order of the pole, such that $U^{m-1}\neq 0$ and $U^m=0$.  If
  $\lambda_0$ is a pole of order $m$, then $\lambda_0$ is an eigenvalue
  since $(\lambda_0 I-A)U^{m-1}=-U^m=0$ and hence $\{0\} \neq
  \im(U^{m-1})\subseteq\ker(\lambda_0 I-A)$.  If $\lambda_0$ is a pole
  of order $m \ge 2$, then there always exists a \emph{generalised
    eigenvector} $x \in \ker((\lambda_0 I-A)^2) \setminus
  \ker(\lambda_0 I-A)$.  Indeed, if we choose $y \in E$ such that
  $U^{m-1}y \not= 0$, then $x = U^{m-2}y$ has the desired properties.

  The dimension of the eigenspace $\ker(\lambda_0 I-A)$ is called the
  \emph{geometric multiplicity} of $\lambda_0$ as an eigenvalue of
  $A$. The dimension of the \emph{generalised eigenspace}
  $\bigcup_{m\in\bbN}\ker(\lambda_0 I-A)^m$ is called the
  \emph{algebraic multiplicity}. The generalised eigenspace coincides
  with $\im P$. The eigenvalue $\lambda_0$ is called
  \emph{geometrically} (\emph{algebraically}) \emph{simple} if its
  geometric (algebraic) multiplicity equals $1$. The pole $\lambda_0$ of
  the resolvent is a simple pole if and only if $\ker(\lambda_0 I-A) =
  \ker((\lambda_0 I - A)^2)$ if and only if $\ker(\lambda_0 I - A) =
  \im P$. Assume that the geometric multiplicity is finite. Then it
  follows that $\lambda_0$ is a simple pole if and only if the algebraic
  and geometric multiplicities coincide.  See also
  \cite[Section~4]{campbell:13:lac} for some details.
\end{remark}
We assume the reader to be familiar with the basic theory of Banach
lattices. As a standard reference for this topic we refer to
\cite{Schaefer1974}. If $E$ denotes a complex Banach lattice, then $E$
is by definition the complexification of a real Banach lattice $E_\bbR$;
see \cite[Section~II.11]{Schaefer1974}. We denote by $E_+ := (E_\bbR)_+
:= \{f \in E_\bbR\colon f\geq 0\}$ the \emph{positive cone} in $E$.  To
avoid any ambiguities we shall adopt the following conventions, which
mirror the notation introduced above when $E=C(K)$: we call $f \in E$
\emph{positive} and write $f\geq 0$ if $f \in E_+$, we write $f > 0$ if
$f \geq 0$ and $f\neq 0$. If $E_+\subseteq E_{\bbR}$ has non-empty
interior, then we write $f\gg 0$ if $f$ is in the interior of $E_+$ and
say $f$ is \emph{strongly positive}.

A linear operator $T$ between two complex Banach lattices $E$ and $F$ is
called \emph{positive} if $TE_+ \subseteq F_+$. If $F_+
\subseteq F_{\mathbb R}$ has non-empty interior we call $T$
\emph{strongly positive} and we write $T\gg 0$ if $Tf\gg0$ whenever
$f>0$. We also apply this notation to elements from the dual space $E'$
and thus say the functional $\varphi\in E'$ is strongly positive if
$\varphi\gg 0$ as a linear map from $E$ into $\mathbb C$, that is,
$\langle\varphi,f\rangle = \varphi(f)>0$ whenever $f>0$.

A linear operator $A\colon E \supset D(A) \to E$ is called \emph{real}
if $x+iy \in D(A)$ implies that $x,y \in D(A)$ whenever $x,y \in
E_\bbR$, and if $A$ maps $D(A) \cap E_\bbR$ into $E_{\mathbb R}$.  Note
that a positive operator $T\in \calL(E)$ is automatically real.

We assume the reader to be familiar with the basic theory of
$C_0$-semigroups.  We will always denote a $C_0$-semigroup on a complex
Banach space $E$ by $(e^{tA})_{t \geq 0}$, where $A$ is the generator of
the semigroup. The \emph{growth bound} of the semigroup $(e^{tA})_{t\ge
  0}$ will be denoted by $\gbd(A)$.  A $C_0$-semigroup $(e^{tA})_{t \geq
  0}$ on a complex Banach lattice is called \emph{real} if the operator 
  $e^{tA}$ is real for every
$t\geq 0$.  It is easy to see that $(e^{tA})_{t\geq 0}$ is real if and
only if $A$ is real.

The following result will be important to reduce some results for
eventually positive semigroups to results for positive semigroups. For
real Banach lattices it is stated in
\cite[Proposition~III.11.5]{Schaefer1974}, but it easily generalises to
complex Banach lattices.
\begin{proposition}
  \label{prop_range_of_pos_proj}
  Let $E$ be a complex Banach lattice and let $P \in \calL(E)$ be a
  positive projection. Then there is an equivalent norm on $PE$ such
  that $PE_\bbR \subseteq PE$ is a real Banach lattice for the order
  induced by $E_\bbR$ and such that $PE$ becomes the complexification of
  the real Banach lattice $PE_\bbR$.
\end{proposition}
We will need the following result on the range of the spectral
projection associated with the peripheral spectrum.
\begin{proposition}
  \label{prop_convergence_on_rangel_of_per_spec_proj}
  Let $(e^{tA})_{t \geq 0}$ be a $C_0$-semigroup on a complex Banach
  space $E$ such that $\sigma(A)\neq \emptyset$.  Suppose that
  $\sigma_{\per}(A)$ as defined in \eqref{eq:sigma-per} is finite and
  consists of simple poles of the resolvent. Denote by $P_{\per}$ the
  spectral projection corresponding to $\sigma_{\per}(A)$.  Then there
  exists a sequence of positive integers $t_n \to \infty$ such that
  $e^{t_n(A-\spb(A) I)}f \to f$ for every $f\in\im(P_{\per} )$.
\end{proposition}
\begin{proof}
  We may assume throughout the proof that $\spb(A) = 0$ and that
  $\sigma_{\per}(A) \not= \emptyset$. Let $f \in P_{\per}E$ and let
  $\sigma_{\per}(A) = \{i\beta_1,\dots,i\beta_m\}$ for
  $\beta_1,\dots,\beta_m \in \bbR$. If $P_k$ denotes the spectral projection
  associated with $i\beta_k$, then $P_{\per} = P_1 + \dots + P_m$. As
  $i\beta_k$ are simple poles of the resolvent we have $\im
  P_k=\ker(i\beta_k I-A)$ for every $k=1,\dots,m$ and therefore
  $e^{tA}f = e^{it\beta_1} P_1f + \dots +e^{it\beta_m}P_mf$ for all
  $t\geq 0$. Let $\lambda=(e^{i\beta_1},\dots,e^{i\beta_m})\in\bbT^m$,
  where $\bbT$ is the unit circle in $\mathbb C$ and $\bbT^m$ the
  standard $m$-dimensional torus. Define $s_\lambda\colon \bbT^m \to
  \bbT^m$ to be the group rotation by $\lambda$, that is, $s_\lambda z =
  \lambda z$ for every $z\in \bbT^m$; here, $\bbT^m$ is endowed with
  pointwise multiplication, with respect to which it is a compact
  topological group.  By a standard result from topological dynamical
  systems the element $\one = (1,\dots,1)$ is \emph{recurrent} with
  respect to the group rotation $s_\lambda$, that is, there exists a
  sequence of positive integers $t_n \to \infty$ such that
  $s_\lambda^{t_n} \one \to \one$; see \cite[Definition~1.2 and
  Theorem~1.2]{Furstenberg1981}. We conclude that $\lambda^{t_n} \to
  \one$ and hence,
  \begin{displaymath}
    e^{t_nA}f 
    = e^{it_n \beta_1} P_1f + \dots +e^{it_n \beta_m}P_mf
    \to P_1f + \dots + P_mf = f
  \end{displaymath}
  as $n\to\infty$ as claimed.
\end{proof}

\section{Strongly positive projections and Perron--Frobenius properties}
\label{sec:projections}
One important feature of positive operators is that the spectral radius
is itself an element of the spectrum. If the operator $T$ is irreducible
and the spectral radius $\spr(T) > 0$ is a pole of the resolvent of $T$,
then the Perron--Frobenius theorem (or Krein--Rutman theorem) asserts
that $\spr(T)$ is an algebraically simple eigenvalue of $T$ and $T'$
with a strongly positive eigenvector; see \cite[Theorem~V.5.2 and its
Corollary and Theorem~V.5.4]{Schaefer1974}. We may refer to this as $T$
having a ``Perron--Frobenius property''. However, there seems to be no intrinsic 
reason why this property should only be considered for the spectral radius of a 
bounded, positive operator. Hence, we start by characterising arbitrary
eigenvalues having a Perron--Frobenius type property. We also explain
this result geometrically.

\begin{proposition}
  \label{prop:projections}
  Let $A$ be a closed, densely defined and real operator on
  $E:=C(K)$. Suppose that $\lambda_0 \in \bbR$ is an eigenvalue of $A$
  and a pole of the resolvent $R(\phdot,A)$. Let $P$ be the spectral
  projection associated with $\lambda_0$. Then the following assertions
  are equivalent.
  \begin{enumerate}[\upshape (i)]
  \item $P\gg 0$;
  \item The eigenvalue $\lambda_0$ of $A$ is geometrically simple, and
    $\ker(\lambda_0 I - A)$ and $\ker(\lambda_0 I-A')$ each contain a
    strongly positive vector;
  \item The eigenvalue $\lambda_0$ of $A$ is algebraically simple,
    $\ker(\lambda_0 I - A)$ contains a strongly positive vector and
    $\im(\lambda_0 I-A)\cap E_+ = \{0\}$.
  \end{enumerate}
  If assertions {\upshape(i)}--{\upshape(iii)} are fulfilled, then
  $\lambda_0$ is a simple pole of the resolvents $R(\phdot,A)$ and
  $R(\phdot,A')$. Moreover, $\dim(\im P)=\dim(\im P')=1$ and $\lambda_0$
  is the only eigenvalue of $A$ having a positive eigenvector.
\end{proposition}

\begin{proof}
  By replacing $A$ with $A-\lambda_0 I$, we may assume without loss of
  generality that $\lambda_0=0$.

  ``(i) $\Rightarrow$ (ii)'' As $P$ is a positive operator its image is
  spanned by positive elements. Let $u,y\in E_+\cap\im P$ be
  non-zero. Since $P\gg 0$ we have $u=Pu\gg 0$ and so
  $\alpha_0:=\inf\{\alpha\colon \alpha u-y\geq 0\} \in (0,\infty)$. As
  $E_+$ is closed $\alpha_0 u-y\geq 0$, and moreover $\alpha_0 u-y\in\im
  P$. Hence, either $\alpha_0u-y=0$ or $\alpha_0u-y\gg 0$. As $E_+$ has
  non-empty interior in $E_{\bbR}$, the second possibility cannot occur
  as it contradicts the definition of $\alpha_0$. Hence $y=\alpha_0u$
  and so $\dim(\im P)=1$. In particular $\lambda_0$ is algebraically and
  hence geometrically simple with eigenfunction $u\gg 0$. Note that $0$
  is also an algebraically simple eigenvalue of $A'$, see
  \cite[Remark~III-6.23]{Kato1976}, and that the dual $P'$ of $P$ is the
  corresponding spectral projection. As $P\gg 0$ also $P'v\gg 0$ for all
  $v>0$ in $E'_+$, that is, $\im P'$ is spanned by an eigenvector $v\gg
  0$ of $A'$ corresponding to the eigenvalue $0$.

  ``(ii) $\Rightarrow$ (iii)'' By assumption $0$ is a geometrically
  simple eigenvalue. To prove that it is algebraically simple it is
  sufficient to show that $x\in\ker A^2$ implies $x\in\ker A$, see
  Remark~\ref{rem:Laurent-exp}. Hence let $x\in\ker A^2$. We know from
  (ii) that $\ker A$ is spanned by a vector $u\gg 0$ and that $\ker A'$
  contains a vector $v\gg 0$. As $Ax\in\ker A$, there exists
  $\alpha\in\mathbb C$, $|\alpha|=1$ such that $\alpha Ax\geq 0$.  As
  $\langle v,\alpha Ax\rangle =\langle A'v,\alpha x\rangle = 0$ and
  $v\gg 0$ we conclude that $Ax = 0$.  Therefore $\langle u\rangle =
  \ker A=\im P$. Now let $y=Ax\in E_+\cap\im A$. Then $\langle
  v,y\rangle=\langle v,Ax\rangle=\langle A'v,x\rangle=0$. As $v\gg 0$
  and $y\geq 0$ we conclude that $y=0$. Hence, $E_+\cap\im A=\{0\}$.

  ``(iii) $\Rightarrow$ (i)'' As $0$ is algebraically simple, (iii)
  implies that $\im P = \ker A=\langle u\rangle$ for some $u \gg 0$ and
  $E=\im P\oplus\im A$. Hence, if $x>0$, then there exist
  $\alpha\in\mathbb C$ and $y\in D(A)$ such that $x=\alpha u+Ay$. Since
  $A$ is real, so is $P$, and hence $\alpha \in \bbR$. If $\alpha\leq
  0$, then $x-\alpha u=Ay\in E_+\cap\im A$, which implies $Ay=0$ by
  (iii). Then $x=\alpha u=Px\leq 0$ and $0<x$, which is a
  contradiction. Hence, $\alpha > 0$ and thus $Px=\alpha u\gg 0$.

  Finally note that the algebraic simplicity of $0$ as an eigenvalue of
  $A$ from (iii) implies that $0$ is a simple pole of the resolvent
  $R(\phdot,A)$; thus, it is also a simple pole of $R(\phdot,A') =
  R(\phdot,A)'$. Now suppose that $\lambda_1\neq 0$ is an eigenvalue of
  $A$ with eigenvector $u_1$. Then $u_1=\lambda_1^{-1}Au_1\in\im A$. As
  $E_+\cap\im A=\{0\}$ by (iii) we conclude that $u_1$ cannot be
  positive. Hence, $0$ is the only eigenvalue having a positive
  eigenvector.
\end{proof}
To the best of our knowledge, the relationship between strong positivity
of the spectral projection and the existence of strongly positive
eigenvectors for arbitrary poles of the resolvent as given in
Proposition~\ref{prop:projections} has not been investigated
before. Our argument which shows that algebraic simplicity follows from 
geometric simplicity if the corresponding eigenvectors of $A$ and $A'$ 
are strongly positive is similar to the proof of \cite[Theorem~4.12(ii)]{Grobler1995}.
Moreover, a related argument for the eigenspace associated with the
spectral bound of a positive semigroup can be found in
\cite[Remark~C-IV-2.2(c)]{Arendt1986}.
\begin{remark}
  \label{rem:projections}
  There is a simple geometric explanation for the equivalent conditions
  in Proposition~\ref{prop:projections} if $E=\mathbb R^N$ and
  $\lambda_0=0$. Due to the algebraic simplicity of the eigenvalue $0$
  the direct sum decomposition $E=\ker(A)\oplus\im(A)$ completely
  reduces $A$. Recall that $Px$ is the projection of $x$ onto the span
  of an eigenvector $u$ of $A$. The fact that $P\gg 0$ means that $u\gg
  0$ and that $E_+$ is on one side of $\im A$. This is also the explicit
  statement in (iii). To interpret (ii) let $v\gg 0$ be an eigenvector
  of $A'$ to the eigenvalue $0$. Then, $\langle v,Ax\rangle=\langle
  A'v,x\rangle=0$ for all $x\in\mathbb R^N$, so that $\im A=\langle
  v\rangle^\perp$ is perpendicular to $v$. Hence, if $E_+$ is to be on
  one side of $\im A$, then $v\gg 0$ (or equivalently $v\ll 0$). The
  configuration is illustrated in Figure~\ref{fig:projection}.
  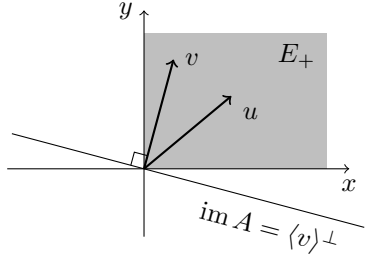
\begin{figure}[ht]
    \centering
    \begin{tikzpicture}[scale=.6,font=\footnotesize]
      \path[fill=lightgray] (0,0) rectangle (4,3) node[below left]
      {$E_+$};%
      \draw[->] (-3,0) -- (4.5,0) node[below] {$x$};%
      \draw[->] (0,-1.5) -- (0,3.5) node[left] {$y$};%
      \draw[->,thick] (0,0) -- (40:2.5) node[below right] {$u$};%
      \begin{scope}[rotate=-15]
        \draw[->,thick] (0,0) -- (0,2.5) node[right] {$v$};%
        \draw (-.3,0) |- (0,.3);%
        \draw (-3,0) -- (5,0) node[below,near end,rotate=-15] {$\im
          A=\langle v\rangle^\perp$};
      \end{scope}
    \end{tikzpicture}
    \caption{Geometric meaning of the Perron--Frobenius property}
    \label{fig:projection}
  \end{figure}
\end{remark}
\begin{remark}
  In \cite{Noutsos2006, Noutsos2008}, a matrix $A$ is said to have the
  \emph{strong Perron--Frobenius property} if $\spr(A)$ is a simple and
  dominant eigenvalue of $A$ having a strongly positive
  eigenvector. However, the full conclusion of the Perron--Frobenius theorem is
  statement (ii) in Proposition~\ref{prop:projections}; see for instance
  \cite[Theorem~1.1]{Seneta1981}. Hence it seems rather more natural to
  define this latter statement to be the ``strong Perron--Frobenius
  property''. According to Proposition~\ref{prop:projections} one could
  then summarise the conclusion of the Perron--Frobenius theorem by
  saying that the spectral projection $P$ associated with $\spr(A)$ is
  strongly positive.
\end{remark}

\section{Eventually strongly positive resolvents}
\label{sec:resolvents}
It is a feature of Perron--Frobenius theory that the resolvent
$R(\lambda,A)$ is positive for all $\lambda>\spb(A)$. In the context of
eventually positive semigroups we cannot expect such a property for all
$\lambda>\spb(A)$, but we show that there is nevertheless a weaker
positivity property. We begin with a definition.
\begin{definition}
  \label{def:evpos-resolvent-c_k}
  Let $A$ be a closed real operator on $E = C(K)$ and let $\lambda_0$ be
  either $-\infty$ or a spectral value of $A$ in $\bbR$.
  \begin{enumerate}[(a)]
  \item The resolvent $R(\phdot,A)$ is called \emph{individually
      eventually} (\emph{strongly}) \emph{positive} \emph{at}
    $\lambda_0$, if there is a $\lambda_2 > \lambda_0$ with the
    following properties: $(\lambda_0, \lambda_2] \subseteq\rho(A)$ and
    for each $f \in E_+ \setminus\{0\}$ there is a $\lambda_1 \in
    (\lambda_0, \lambda_2]$ such that $R(\lambda,A)f \geq 0$ ($\gg 0$)
    for all $\lambda \in (\lambda_0, \lambda_1]$.
  \item The resolvent $R(\phdot,A)$ is called \emph{uniformly
      eventually} (\emph{strongly}) \emph{positive} \emph{at}
    $\lambda_0$, if there exists $\lambda_1 > \lambda_0$ with the
    following properties: $(\lambda_0, \lambda_1] \subseteq\rho(A)$ and
    $R(\lambda,A) \geq 0$ ($\gg 0$) for every $\lambda \in (\lambda_0,
    \lambda_1]$.
  \end{enumerate}
\end{definition}
Example~\ref{ex_ind_evtl_pos_does_not_imply_unif_evtl_pos} below shows
that it is necessary to distinguish between individual and uniform
properties in Definition~\ref{def:evpos-resolvent-c_k}.

Positivity of the resolvent is an important concept within the theory of
positive semigroups. For example, it is well known that a
$C_0$-semigroup $(e^{tA})_{t \ge 0}$ is positive if and only if
$R(\lambda,A) \ge 0$ for all $\lambda > \spb(A)$. If this is the case,
then we even have $R(\lambda,A) \gg 0$ for some (equivalently all)
$\lambda > \spb(A)$ if and only if the semigroup is irreducible, see
\cite[Definition~B-III-3.1]{Arendt1986} (note however that $R(\lambda,A)
\gg 0$ for some $\lambda > \spb(A)$ implies $R(\lambda,A) \gg 0$ for all
$\lambda > \spb(A)$ only if it is already known that $(e^{tA})_{t \ge
  0}$ is a positive semigroup). The assertions in
Definition~\ref{def:evpos-resolvent-c_k} are thus generalisations of
those well-known resolvent properties to spectral values other than
$\spb(A)$ and to smaller ranges of $\lambda$.

We proceed by stating a simple criterion for uniform eventual positivity
of the resolvent at some spectral value $\lambda_0$.
\begin{proposition}
  \label{prop_evtl_pos_res_at_one_point_on_c_k}
  Let $A$ be a closed real operator on $E=C(K)$. Let $\lambda_0$ be
  $-\infty$ or a spectral value of $A$ in $\bbR$. Assume that
  there exists $\lambda_1>\lambda_0$ such that $(\lambda_0, \lambda_1]
  \subseteq \rho(A)$ and $R(\lambda_1,A)\geq 0$. Then the following
  assertions are true.
  \begin{enumerate}[\upshape (i)]
  \item The resolvent $R(\phdot,A)$ is uniformly eventually positive at
    $\lambda_0$. More precisely, $R(\lambda,A)\geq 0$ for all
    $\lambda\in(\lambda_0,\lambda_1]$.
  \item If $R(\lambda_1,A)^n$ is strongly positive for some $n\in\mathbb
    N$, then $R(\phdot,A)$ is uniformly eventually strongly positive at
    $\lambda_0$. More precisely, $R(\lambda,A)\gg 0$ for all
    $\lambda\in(\lambda_0,\lambda_1)$.
  \end{enumerate}
\end{proposition}
\begin{proof}
  We prove (i) and (ii) simultaneously. To that end let $\prec$ denote
  $\leq$ in case (i) and $\ll$ in case (ii). We set
  \begin{math}
    U:=\bigl\{\lambda\in(\lambda_0, \lambda_1)\colon
    \text{$R(\mu,A)\succ 0$ for all $\mu\in(\lambda,\lambda_1)$}\bigr\}
  \end{math}
  and show that $U=(\lambda_0,\lambda_1)$.  Because
  $(\lambda_0,\lambda_1)$ is connected it is sufficient to show that $U$
  is non-empty, open and closed in $(\lambda_0,\lambda_1)$.
  
  If $\lambda_n\in U$ with $\lambda_n\to\lambda\in
  (\lambda_0,\lambda_1)$ and $\mu>\lambda$, then $\mu>\lambda_n$ for $n$
  large enough.  Hence, by definition of $U \ni \lambda_n$, we have
  $R(\mu,A)\succ 0$ for all $\mu\in(\lambda,\lambda_1)$. Thus
  $\lambda\in U$, showing that $U$ is closed in $(\lambda_0,\lambda_1)$.
  
  Given $\mu_0$ in the open set $\varrho(A)\subseteq\bbC$, the analytic
  function $R(\phdot,A)$ can be expanded as a power series
  \begin{equation}
    \label{form_series_repr_res_on_right_real_line}
    R(\mu,A)
    =\sum_{k=0}^\infty (\mu_0-\mu)^k R(\mu_0,A)^{k+1}
  \end{equation}
  whenever $\varepsilon>0$ and $\mu$ are such that $\mu\in
  B(\mu_0,\varepsilon)\subseteq\rho(A)$; see
  \cite[Proposition~IV.1.3(i)]{Engel2000}. If we choose
  $\mu_0=\lambda_1$, then $R(\mu_0,A)\geq 0$ by assumption. In case (ii)
  at least one of the terms in
  \eqref{form_series_repr_res_on_right_real_line} is strongly
  positive if $\mu<\mu_0$. Hence in both cases
  $(\lambda_1-\varepsilon,\lambda_1)\subseteq U$ and so
  $U\neq\emptyset$.
  
  To show that $U$ is open let $\lambda\in U$ and let $\varepsilon>0$
  such that $B(\lambda,\varepsilon)\subseteq\rho(A)$. Then choose
  $\mu_0>\lambda$ and $\varepsilon_0>0$ such that $\mu_0 \in U$ and
  $\lambda\in B(\mu_0,\varepsilon_0)\subseteq
  B(\lambda,\varepsilon)$. As $R(\mu_0,A)\succ 0$ it follows from
  \eqref{form_series_repr_res_on_right_real_line} that $R(\mu,A)\succ 0$
  for all $\mu\in(\mu_0-\varepsilon_0,\mu_0)$. By choice of $\mu_0$ and
  $\varepsilon$ the interval $(\mu_0-\varepsilon_0,\mu_0)$ is a
  neighbourhood of $\lambda$, showing that $U$ is open in
  $(\lambda_0,\lambda_1)$. Thus, $U=(\lambda_0,\lambda_1)$ as claimed.
\end{proof}
A consequence of Proposition~\ref{prop_evtl_pos_res_at_one_point_on_c_k}
is the following simple but useful criterion for eventual positivity of
resolvents, which we will use in
Section~\ref{section_applications_on_c_k} to study the square of the
Robin Laplacian.
\begin{proposition}
  \label{prop_square_of_resolvent_pos_operator_c_k}
  Let $A$ be a closed and real operator on $E=C(K)$. Let $\lambda_0 < 0$
  be $-\infty$ or a spectral value of $A$ and assume that $(\lambda_0,0]
  \subseteq\rho(A)$.  Furthermore, suppose that there is a closed
  operator $B\colon C(K) \supseteq D(B) \to C(K)$ such that
  $A=(iB)^2=-B^2$. If $R(0,B)$ is (strongly) positive, then $R(\phdot,
  A)$ is uniformly eventually (strongly) positive at $\lambda_0$.
\end{proposition}
\begin{proof}
  Both assertions follow from
  Proposition~\ref{prop_evtl_pos_res_at_one_point_on_c_k} since $R(0, A)
  = R(0,B)^2$.
\end{proof}
We now formulate the main result of this section and characterise
eventually positive resolvents by means of positive projections. In
conjunction with Proposition~\ref{prop:projections} the following
theorem not only contains a Perron--Frobenius (or Krein--Rutman) type
theorem but also its converse.

\begin{theorem}
  \label{thm:evpos-resolvents-c_k}
  Let $A$ be a closed, densely defined and real operator on $E = C(K)$.
  Suppose that $\lambda_0\in\mathbb R$ is an eigenvalue of $A$ and a
  pole of the resolvent $R(\phdot,A)$. Let $P$ be the corresponding
  spectral projection. Then the following assertions are equivalent.
  \begin{enumerate}[\upshape (i)]
  \item $P\gg 0$.
  \item The resolvent $R(\phdot,A)$ is individually eventually strongly
    positive at $\lambda_0$.
  \end{enumerate}
  If $\lambda_0 = \spb(A)$, then {\upshape (i)} and {\upshape (ii)} are
  also equivalent to the following assertions.
  \begin{enumerate}[\upshape (i)]
    \setcounter{enumi}{2}
  \item For every $\lambda > \spb(A)$ and every $f > 0$ there exists
    $n_0 \in \bbN$ such that $R(\lambda,A)^n f \gg 0$ for all $n \geq
    n_0$.
  \item There exists $\lambda>\spb(A)$ such that for every $f > 0$ there
    exists $n_0 \in \bbN$ such that $R(\lambda,A)^n f \gg 0$ for all $n
    \geq n_0$.
  \end{enumerate}
\end{theorem}
\begin{remark}
  \label{rem:eigenvalue-shift}
  Let the assumptions of Theorem~\ref{thm:evpos-resolvents-c_k} be
  satisfied, and suppose that the spectral projection $P$ associated
  with $\lambda_0$ is strongly positive. 
  Theorem~\ref{thm:evpos-resolvents-c_k} shows that
  $R(\phdot,A)$ is individually eventually strongly positive at
  $\lambda_0$. We show that $\lambda_0$ can be anywhere on the real
  axis, independently of the spectral bound $\spb(A)$. For
  $\mu\in\mathbb R$ we set $A_\mu:=A-\mu P$. Then $\sigma(A_\mu)
  =\bigl(\sigma(A)\setminus\{\lambda_0\}\bigr)\cup\{\lambda_0-\mu\}$. Moreover,
  if $\lambda_0 - \mu \not\in \sigma(A)$, then the spectral projection
  associated with $\lambda_0-\mu$ is still $P$ and hence
  $R(\phdot,A_\mu)$ is individually eventually strongly positive at
  $\lambda_0-\mu$.

  A similar argument can also be used to show that $\spb(A)$ does not
  need to be a spectral value of $A$ even if $R(\lambda,A)\geq 0$ in
  some right neighbourhood of $\spb(A)$. We refer to
  Remark~\ref{rem:sg-rot}(b) for details.
\end{remark}
Assertions (iii) and (iv) in Theorem~\ref{thm:evpos-resolvents-c_k} give
conditions on large powers of the resolvent. In fact, it is sufficient
to consider a single power, provided that this power is a strongly
positive operator.

\begin{proposition}
  \label{prop_a_single_power_of_res_positive}
  Let $A$ be a closed, densely defined and real operator on $E =
  C(K)$. Suppose that $\spb(A)\in\mathbb R$ is an eigenvalue of $A$ and
  a pole of the resolvent $R(\phdot,A)$, and that there exist $n\in\bbN$
  and $\lambda_1>\spb(A)$ such that $R(\lambda_1,A)^n\gg 0$. Then the
  equivalent statements of Theorem~\ref{thm:evpos-resolvents-c_k} are
  fulfilled for $\lambda_0 = \spb(A)$.
\end{proposition}
  
The reader might compare the above proposition with
Proposition~\ref{prop_evtl_pos_res_at_one_point_on_c_k}. If, in addition
to the assumptions of
Proposition~\ref{prop_a_single_power_of_res_positive}, the operator
$R(\lambda_1,A)$ is positive, then $R(\phdot,A)$ is \emph{uniformly}
eventually strongly positive.

The remainder of this section is devoted to the proofs of
Theorem~\ref{thm:evpos-resolvents-c_k} and
Proposition~\ref{prop_a_single_power_of_res_positive}. We can assume
without loss of generality that $\lambda_0=0$ by replacing $A$ with
$A-\lambda_0 I$.

The first step towards the proof of
Theorem~\ref{thm:evpos-resolvents-c_k} is to express the spectral
projection $P$ in terms of the resolvent.
\begin{lemma}
  \label{lem_convergence_of_resolvent_to_spectral_projection}
  Let $A$ be a closed operator on a complex Banach space $E$ and assume
  that $0$ is an eigenvalue of $A$ and a pole of the resolvent
  $R(\phdot,A)$. Denote by $P$ the corresponding spectral projection.
  \begin{enumerate}[\upshape (i)]
  \item If $\lambda > 0$ is contained in $\rho(A)$ and if the operator
    family $\bigl([\lambda R(\lambda,A)]^n\bigr)_{n \in \bbN}$ is
    bounded, then $0$ is a simple pole of the resolvent.
  \item Suppose in addition that $0 = \spb(A)$. If $\lambda>0$ and
    $\spb(A) = 0$ is a simple pole of the resolvent, then
    \begin{equation}
      \label{from_res_powers_converge_to_spec_proj}
      \lim_{n \to \infty} [\lambda R(\lambda,A)]^n = P
    \end{equation}
    in $\mathcal L(E)$.
  \end{enumerate}
\end{lemma}
\begin{proof}
  (i) We give a proof by contrapositive. Suppose that $0$ is a pole of
  order $m \ge 2$ of $R(\phdot,A)$. If we set $T:=\lambda R(\lambda,A)$,
  then $1$ is a pole of order $m$ of $R(\phdot,T)$; see
  \cite[Proposition~IV.1.18]{Engel2000}. As $m \ge 2$ there is a
  generalised eigenvector $x \in \ker( I-T)^2 \setminus \ker( I-T)$; see
  Remark~\ref{rem:Laurent-exp}.  A short induction argument now shows
  that $T^nx = x-n( I-T)x$ for each $n \in \bbN_0$.  As $( I-T)x \neq
  0$, this implies that $(T^n)_{n \in \bbN}$ is unbounded.
  
  (ii) By the spectral mapping theorem for resolvents, we have
  \begin{displaymath}
    \sigma\bigl(\lambda R(\lambda,A)\bigr) \setminus \{0\}
    =\frac{\lambda}{\lambda-\sigma(A)}
  \end{displaymath}
  for all $\lambda>0$; see \cite[Theorem~III-6.15]{Kato1976}. The map
  $\mu\to\dfrac{\lambda}{\lambda-\mu}$ is a M\"obius transformation
  mapping the left half plane onto the disc $B_{1/2}(1/2)$. As $1$ is an
  isolated point of the spectrum of $\lambda R(\lambda,A)$, there exists
  $c\in(0,1)$ such that $|\mu|\leq c$ for all $\mu\in
  \sigma\bigl(\lambda R(\lambda,A)\bigr)\setminus\{1\}$. In particular,
  \begin{equation}
    \label{from_spectral_radius_for_normalised_resolvent}
    \spr\bigl(\lambda R(\lambda,A)(I-P)\bigr)\leq c<1.
  \end{equation}
  Since $\spb(A) = 0$ is a first order pole of the resolvent we have
  $\im P=\ker A$ and thus $\lambda R(\lambda,A)P = P$. As $\im
  P\oplus\ker P$ completely reduces $\lambda R(\lambda,A)$, we obtain
  \begin{displaymath}
    \bigl[\lambda R(\lambda,A)\bigr]^n =\bigl[\lambda
    R(\lambda,A)P\bigr]^n+
    \bigl[\lambda R(\lambda,A)(I-P)\bigr]^n
    = P + \bigl[\lambda R(\lambda,A)(I-P)\bigr]^n.
  \end{displaymath}
  Hence $\bigl[\lambda R(\lambda,A)\bigr]^n \to P$ as $n \to \infty$ due
  to \eqref{from_spectral_radius_for_normalised_resolvent}.
\end{proof}
\begin{lemma}
  \label{lem_powers_of_operator_positive_eigenvector}
  Let $T \in \calL(E)$ be an operator on a complex Banach lattice $E$
  with spectral radius $\spr(T) = 1$. Suppose that for every $x\geq 0$,
  there exists $n_0\in\bbN$ such that $T^nx \geq 0$ for all $n \geq
  n_0$. If $1$ is a pole of the resolvent $R(\phdot, T)$, then the
  eigenspace $\ker( I - T)$ contains a positive, non-zero element.
\end{lemma}
\begin{proof}
  Let $m$ be the order of $1$ as a pole of $R(\phdot, T)$. From the
  Laurent expansion \eqref{eq:Laurent-exp} we have that
  \begin{displaymath}
    \lim_{\lambda \downarrow 1}(\lambda-1)^mR(\lambda,T)=U^{m-1}
  \end{displaymath}
  in $\calL(E)$.  Now, let $\lambda>1$ and $0\leq x\in E$.  Then there
  exists $n_0\in\mathbb N$ such that $T^nx\geq 0$ for all $n\geq
  n_0$. Hence we obtain from the Neumann series representation of the
  resolvent that
  \begin{displaymath}
    R(\lambda,T)x
    =\sum_{k=0}^\infty\frac{T^k}{\lambda^{k+1}}x
    \geq \sum_{k=0}^{n_0-1}\frac{T^k}{\lambda^{k+1}}x
  \end{displaymath}
  whenever $|\lambda| > 1$. In particular,
  \begin{displaymath}
    U^{m-1}x
    =\lim_{\lambda \downarrow 1} (\lambda - 1)^m R(\lambda,A)x
    \geq\lim_{\lambda \downarrow 1} (\lambda - 1)^m
    \sum_{k=0}^{n_0-1}\frac{T^k}{\lambda^{k+1}}x
    =0 \text{.}
  \end{displaymath}
  Hence $U^{m-1}$ is a positive operator. By
  Remark~\ref{rem:Laurent-exp} $\im(U^{m-1})$ is non-trivial and
  consists of eigenvalues of $T$. As $U^{m-1}$ is positive there exists
  a positive eigenvector.
\end{proof}
The above lemma suggests that it could be interesting to develop a 
theory of ``eventually positive'' or ``power-positive'' operators
similar to \cite{Noutsos2006} in infinite dimensions. However, since we
are concerned with semigroups we will only use
Lemma~\ref{lem_powers_of_operator_positive_eigenvector} as a technical
tool for the proof of Theorem~\ref{thm:evpos-resolvents-c_k}.

In order to prove Theorem~\ref{thm:evpos-resolvents-c_k}, one might be
tempted to try to use the classical Perron--Frobenius theory, assuming
that $R(\lambda,A)\gg 0$ for some $\lambda>0$. However, here we only
assume that the resolvent is \emph{individually} eventually positive, so
we need to consider properties of families of operators having a certain
weaker pointwise eventual positivity property.
\begin{lemma}
  \label{lem_evtl_pos_family_of_op_eigenvectors}
  Let $E=C(K)$. Let $(J,\preceq)$ be a non-empty totally ordered set and
  let $\calT:=(T_j)_{j \in J}$ be a family in $\calL(E)$ with fixed
  space
  \begin{displaymath}
    F:=\{x \in E\colon\text{$T_jx = x$ for all $j \in J$}\}.
  \end{displaymath}
  Assume that for every $x>0$ there exists $j_x \in J$ such that $T_j x
  \gg 0$ for all $j \succeq j_x$.
  \begin{enumerate}[\upshape (i)]
  \item If the family $(T_j)_{j \preceq j_0}$ is bounded in $\calL(E)$
    for every $j_0\in J$ and $F$ contains an element $x_0>0$, then
    $\calT$ is bounded in $\calL(E)$.
  \item Let $P>0$ be a projection on $E$. If every $T \in \calT$ leaves
    $\ker(P)$ invariant and $\im(P)\subseteq F$, then $P\gg 0$.
  \end{enumerate}
\end{lemma}
\begin{proof}
  (i) It suffices to show that the orbit $(T_jx)_{j \in J}$ is bounded
  for every $x>0$, since then it follows that $(T_jx)_{j \in J} $ is
  bounded for every $x \in E$, which in turn implies (i) due to the
  uniform boundedness principle.
    
  Fix $x > 0$. By assumption there exists $x_0\in F$ such that
  $x_0>0$. Hence, $x_0=T_jx_0\gg 0$ for some $j\in J$, so in particular
  $x_0\gg 0$.  As $x_0\gg 0$ there exists $c>0$ such that $cx_0 \pm x
  \geq 0$.  Hence there exists $j_0\in J$ such that $T_j(cx_0 \pm x)
  \geq 0$ for all $j \succeq j_0$ and thus $|T_j x| \le cx_0$ for all $j
  \succeq j_0$. As $(T_j)_{j\preceq j_0}$ is bounded, this shows that
  $(T_jx)_{j\in J}$ is indeed bounded.

  (ii) If $x>0$, then by assumption $Px\in F$ and $Px \geq 0$. If
  $Px>0$, then by assumption there exists $j\in J$ so that $0\ll
  T_jPx=Px$. Hence, for every $x>0$ either $Px\gg 0$ or $Px=0$.  Let us
  now show that $Px \neq 0$ whenever $x>0$.  As $P$ is non-zero, there
  exists $y>0$ such that $u:=Py > 0$. If $x > 0$, then we find a $j \in
  J$ such that $T_jx \gg 0$, and thus $T_j x - cu \ge 0$ for some $c >
  0$. Since $P$ is positive, we conclude that $PT_j x \ge Pcu = cu >
  0$. In particular, $PT_j x \neq 0$. As $T_j$ leaves $\ker P$
  invariant, we must also have $Px \not= 0$.
\end{proof}

\begin{proof}[Proof of Theorem~\ref{thm:evpos-resolvents-c_k}] We may
  assume that $\lambda_0=0$.
  
  ``(i) $\Rightarrow$ (ii)'' As $P\gg 0$ it follows from
  Proposition~\ref{prop:projections} that $\lambda_0=0$ is a simple pole
  of $R(\phdot,A)$. Given $f>0$ it follows from (i) that $Pf\gg 0$. By
  the Laurent expansion \eqref{eq:Laurent-exp} we conclude that $\lambda
  R(\lambda,A) Pf\to Pf$ in $E$ as $\lambda \downarrow 0$.  As the
  interior of $E_+$ in $E_{\mathbb R}$ is non-empty and $A$ is real it
  follows that there exists $\lambda_0>0$ such that $R(\lambda,A)f\gg 0$
  for all $\lambda\in(0,\lambda_0]$. Hence, $R(\phdot,A)$ is
  individually eventually strongly positive at $\lambda_0$, proving
  (ii). A similar argument using
  Lemma~\ref{lem_convergence_of_resolvent_to_spectral_projection}(ii)
  shows that (i) also implies (iii) if $\lambda_0=\spb(A)=0$.

  ``(ii) $\Rightarrow$ (i)'' We first show that $0$ is a simple pole and
  an eigenvalue with eigenvector $u>0$. Assumption (ii) implies that
  there exists $\lambda_2 > 0$ with the following properties:
  $(0,\lambda_2] \subseteq\rho(A)$ and for every $f>0$ there exists
  $\lambda_1\in(0,\lambda_2]$ such that $R(\lambda,A)f\gg 0$ for all
  $\lambda\in(0,\lambda_1]$. Let $m$ be the order of $0$ as a pole of
  $R(\lambda,A)$. Using the Laurent expansion \eqref{eq:Laurent-exp} we
  see that
  \begin{displaymath}
    U^{m-1}f=\lim_{\lambda\downarrow 0}\lambda^mR(\lambda,A)f\geq 0.
  \end{displaymath}
  In particular $U^{m-1}>0$ and hence $A$ has an eigenvector $u>0$
  corresponding to $0$ by Remark~\ref{rem:Laurent-exp}. Also $\lambda
  R(\lambda,A)u=u$ for all $\lambda>0$, so $u$ is in the fixed space $F$
  of the operator family
  \begin{equation}
    \label{eq:family-res-lambda}
    \calT:=\bigl(\lambda R(\lambda, A)\bigr)_{\lambda \in (0,\lambda_2]}
  \end{equation}
  where the order $\preceq$ on $J:= (0,\lambda_2]$ is given by $\ge$.
  Clearly, the conditions of
  Lemma~\ref{lem_evtl_pos_family_of_op_eigenvectors}(i) are satisfied
  and so the family \eqref{eq:family-res-lambda} is bounded.  Therefore
  $\lambda^kR(\lambda,A)\to 0$ as $\lambda\downarrow 0$ for every $k\geq
  2$ and hence $0$ is a simple pole of $R(\phdot,A)$. From the above
  argument $U^0=P>0$. Moreover $\im P=F$ and thus
  Lemma~\ref{lem_evtl_pos_family_of_op_eigenvectors}(ii) implies (i).
  
  For the rest of the proof we assume that $\lambda_0=\spb(A)=0$.
  Obviously (iii) implies (iv).  ``(iv) $\Rightarrow$ (i)'' We proceed
  similarly as in the previous paragraph and first show that $0$ is a
  simple pole of $R(\phdot,A)$. Let $\lambda>0$ be such that for every
  $f>0$ there exists $n_0\in\mathbb N$ such that
  \begin{equation}
    \label{eq:R-power-positive}
    \bigl[R(\lambda,A)\bigr]^nf\gg 0
  \end{equation}
  for all $n\geq n_0$. The operator $\lambda R(\lambda,A)$ has spectral
  radius $1$, and $1$ is a pole of the resolvent $R(\phdot, \lambda
  R(\lambda,A))$, see \cite[Proposition~IV.1.18]{Engel2000}. By
  Lemma~\ref{lem_powers_of_operator_positive_eigenvector} $\lambda
  R(\lambda,A)$ has an eigenvector $u>0$ for the eigenvalue $1$. In
  particular, $u>0$ is in the fixed space $F$ of the family
  \begin{equation}
    \label{eq:family-res-power}
    \calT:=\big([\lambda R(\lambda,A)]^n\big)_{n \in \bbN},
  \end{equation}
  where the order $\preceq$ on $J:= \mathbb N$ is given by
  $\leq$. Clearly, the conditions of
  Lemma~\ref{lem_evtl_pos_family_of_op_eigenvectors}(i) are satisfied,
  so $\calT$ is bounded.  Now
  Lemma~\ref{lem_convergence_of_resolvent_to_spectral_projection}(i)
  implies that $0$ is a simple pole of $R(\phdot, A)$. Therefore, $\ker
  A = \im P=F$. Using that $\spb(A) = 0$ we can apply
  Lemma~\ref{lem_convergence_of_resolvent_to_spectral_projection}(ii)
  and together with \eqref{eq:R-power-positive} we hence obtain
  $Pf=\lim_{n\to\infty} \bigl[R(\lambda,A)\bigr]^nf\geq 0$ for every
  $f\geq 0$, so $P\geq 0$. Since $0\neq u\in F=\im P$ we even have $P >
  0$.  Now Lemma~\ref{lem_evtl_pos_family_of_op_eigenvectors}(ii)
  implies~(i).
\end{proof}

\begin{proof}[Proof of
  Proposition~\ref{prop_a_single_power_of_res_positive}]
  We may assume that $\lambda_0 = \spb(A) = 0$. If we set
  $T:=\bigl[\lambda R(\lambda,A)\bigr]^n$, then $\spr(T)=1\in\sigma(T)$
  and Lemma~\ref{lem_powers_of_operator_positive_eigenvector} guarantees
  the existence of a positive fixed vector of $T$ (alternatively, we
  could apply the classical Perron--Frobenius
  theory). Lemma~\ref{lem_evtl_pos_family_of_op_eigenvectors}(i) applied
  to the operator family $(T^j)_{j \in \bbN}$ now implies that $T$ is
  power-bounded. Therefore, $\lambda R(\lambda,A)$ is power-bounded as
  well, and so $0$ is a simple pole of $R(\phdot,A)$ by
  Lemma~\ref{lem_convergence_of_resolvent_to_spectral_projection}(i).

  If $P$ denotes the spectral projection of $A$ corresponding to $0$,
  then by
  Lemma~\ref{lem_convergence_of_resolvent_to_spectral_projection}(ii) we
  have $P = \lim_{j \to \infty} T^j$. As $T\gg 0$ we have $P\geq 0$.
  Since $\lambda_0 = \spb(A) = 0$ is an eigenvalue of $A$, we also have
  $P \not= 0$, so $P > 0$. Moreover, $\im(P)=\ker(A)$ because $0$ is a
  simple pole of $R(\phdot,A)$. Thus, $\im(P)$ is contained in the fixed
  space of the operator family $(T^j)_{j \in \bbN}$ and we can apply
  Lemma~\ref{lem_evtl_pos_family_of_op_eigenvectors}(ii) to conclude
  that $P\gg 0$. Hence the equivalent conditions of
  Theorem~\ref{thm:evpos-resolvents-c_k} are fulfilled.
\end{proof}

\section{Eventually strongly positive semigroups}
\label{sec:semigroups}

In this section we come to the heart of the subject. We now use the
results of the previous sections to analyse eventual positivity
properties of $C_0$-semigroups. Let us start by defining the central
notion of this article.

\begin{definition}
  \label{def_eventual_positivity_for_sg_on_c_k}
  Let $(e^{tA})_{t \geq 0}$ be a real $C_0$-semigroup on $E = C(K)$.
  \begin{enumerate}[(a)]
  \item The semigroup $(e^{tA})_{t \geq 0}$ is called \emph{individually
      eventually} (\emph{strongly}) \emph{positive} if for every $f \in
    E_+ \setminus\{0\}$ there exists $t_0 \geq 0$ such that $e^{tA}f
    \geq 0$ ($\gg 0$) for all $t\geq t_0$.
  \item The semigroup $(e^{tA})_{t \geq 0}$ is called \emph{uniformly
      eventually} (\emph{strongly}) \emph{positive} if there exists $t_0
    \geq 0$ such that $e^{tA} \geq 0$ ($\gg 0$) for all $t\geq t_0$.
  \end{enumerate}
\end{definition}

Again, we point out that individual and uniform eventual positivity are
not equivalent, see
Examples~\ref{ex_ind_evtl_pos_does_not_imply_unif_evtl_pos} and
\ref{ex_ind_evtl_pos_does_not_imply_unif_evtl_pos_for_compact_res}
below.

It is well known that a $C_0$-semigroup $(e^{tA})_{t \geq 0}$ on a
Banach lattice is positive if and only if the resolvent $R(\lambda,A)$
is positive for all $\lambda > \spb(A)$. However, the situation is more
complicated for eventual positivity. The point is that the long time
behaviour of the semigroup may be influenced by non-real elements of
the peripheral spectrum, while those spectral values have only a minor
influence on the behaviour of the resolvent $R(\lambda,A)$ as $\lambda
\downarrow \spb(A)$. We illustrate the problem with an example in 
three-dimensional space.

\begin{example}
  \label{ex:sg-rot}
  Let $A$ be the $3\times 3$ matrix generating the rotation semigroup
  $(e^{tA})_{t \geq 0}$ rotating vectors about the line in the direction
  of the unit vector $u_1=3^{-1/2}(1,1,1)$ in $\bbR^3$ (more precisely,
  we consider the extension of this semigroup to $\bbC^3$). Then
  $\sigma(A)=\{0,i,-i\}$. Clearly the spectral projection $P$ associated
  with $\spb(A)=0$ is given by $Px=\langle u_1,x\rangle u_1$. Hence
  $P\gg 0$ and Theorem~\ref{thm:evpos-resolvents-c_k} implies that
  $R(\lambda,A)$ is individually eventually strongly positive at
  $\spb(A) = 0$. However, there exist arbitrarily large $t>0$ such that
  $e^{tA}e_k\not\geq 0$, where $(e_k)$ is the standard basis. Hence,
  $e^{tA}$ cannot be eventually positive.
\end{example}
\begin{remark}
  \label{rem:sg-rot}
  (a) If we modify the above example in such a way that the semigroup
  becomes exponentially stable on the orthogonal complement of $u_1$,
  then it becomes eventually strongly positive. More precisely, for
  $\mu>0$ we consider the generator $\tilde A_\mu:=A-\mu(I-P)$. Then,
  $\sigma(\tilde A_\mu)=\{0,-\mu+i,-\mu-i\}$, that is, $0$ is a dominant
  eigenvalue of $A_\mu$. We then have
  \begin{displaymath}
    e^{t\tilde A_\mu}=P+e^{-\mu t}e^{tA}(I-P)\to P
  \end{displaymath}
  as $t\to\infty$. As $P\gg 0$ it is obvious that $e^{t\tilde A_\mu}\gg
  0$ for $t$ sufficiently large.

  (b) Alternatively, we could modify the above example in the following
  way: as in Remark~\ref{rem:eigenvalue-shift} we let $A_\mu := A - \mu
  P$. We then have $\sigma(A_\mu) = \{-\mu,i,-i\}$.  If $\mu > 0$, then
  $\spb(A_\mu)=0$ and we can make the following observation: We know
  that $R(\lambda,A)$ is uniformly eventually strongly positive at $0$,
  that is, there exists $\lambda_1>0$ such that $R(\lambda_1,A)\gg
  0$. As $R(\lambda_1,A_\mu)$ is a continuous function of $\mu\in\mathbb
  R$ there exists $\mu>0$ such that $R(\lambda_1,A_\mu)\gg 0$. By
  Proposition~\ref{prop_evtl_pos_res_at_one_point_on_c_k} we have
  $R(\lambda,A_\mu)\gg 0$ for all $\lambda\in(-\mu,\lambda_1]$.  In
  particular $R(\lambda,A_\mu)\gg 0$ for all $\lambda\in(0,\lambda_1)$,
  but $0=\spb(A_\mu)\not\in\sigma(A_\mu)$.  This shows that we cannot
  conclude that $\spb(A_\mu)$ is a spectral value of $A_\mu$ if
  $R(\lambda,A_\mu) \ge 0$ for all $\lambda$ in a right neighbourhood of
  $\spb(A_\mu)$.
\end{remark}
Part (a) of the above remark suggests that if $\spb(A)$ is a dominant
eigenvalue, then eventual strong positivity of the semigroup is
equivalent to eventual strong positivity of the resolvent at $\spb(A)$.
Recall from Theorem~\ref{thm:evpos-resolvents-c_k} that individual
eventual strong positivity of the resolvent at a pole $\lambda_0 \in
\sigma(A) \cap \bbR$ of the resolvent has several equivalent
manifestations. The most convenient is that the spectral projection $P$
associated with $\lambda_0$ is strongly positive; this property can also
be characterised by the conditions in
Proposition~\ref{prop:projections}. Thus,
Theorem~\ref{thm:evpos-resolvents-c_k} and
Proposition~\ref{prop:projections} yield various possibilities to check
the second part of condition (ii) in the following theorem.

\begin{theorem}
  \label{thm:evpos-semigroups-c_k}
  Let $(e^{tA})_{t \geq 0}$ be a real $C_0$-semigroup on $E=C(K)$ with
  $\sigma(A) \not= \emptyset$. Suppose that the peripheral spectrum
  given by \eqref{eq:sigma-per} is finite and consists of poles of the
  resolvent. Then the following assertions are equivalent:
  \begin{enumerate}[\upshape (i)]
  \item The semigroup $(e^{tA})_{t\geq 0}$ is individually eventually
    strongly positive.
  \item The semigroup $(e^{t(A-\spb(A) I)})_{t\geq 0}$ is bounded,
    $\spb(A)$ is a dominant spectral value of $A$ and the associated
    spectral projection $P$ fulfils $P\gg 0$.
  \item The semigroup $e^{t(A - \spb(A) I)}$ converges strongly to some
    operator $Q\gg 0$ as $t\to\infty$.
  \end{enumerate}
  If assertions {\upshape(i)}--{\upshape(iii)} are fulfilled, then $Q$
  is the spectral projection associated with $\spb(A)$, that is, $Q=P$.
\end{theorem}
\begin{proof}
  We may assume that $\spb(A)=0$.

  ``(i) $\Rightarrow$ (ii)'' It follows from
  Theorem~\ref{thm:spb-in-spectrum} and
  Theorem~\ref{thm:boundary-spectrum}(i) below that $\spb(A)=0$ is an
  eigenvalue of $A$ admitting an eigenvector $x>0$. As $e^{tA}x=x$ for
  all $t>0$, the vector $x>0$ belongs to the fixed space $F$ of the
  operator family
  \begin{equation}
    \label{eq:sg-family}
    \calT=(e^{tA})_{t \in [0,\infty)}\text{.}
  \end{equation}
  Moreover, by (i), for every $f>0$ there exists $t_0>0$ such that
  $e^{tA}f\gg 0$ for all $t\geq t_0$. Finally, by strong continuity
  the sub-family $(e^{tA})_{t \in [0,T]}$ is bounded for every
  $T>0$. Applying Lemma~\ref{lem_evtl_pos_family_of_op_eigenvectors}(i)
  to the operator family \eqref{eq:sg-family} with $J=[0,\infty)$ and
  the order $\preceq$ given by $\leq$, we conclude that the semigroup
  $(e^{tA})_{t \in [0,\infty)}$ is bounded.

  We next show that $\spb(A)$ is a simple pole of $R(\phdot,A)$. Let
  $C:=\sup_{t\geq 0}\|e^{tA}\|$. By the Laplace transform representation
  of $R(\lambda,A)$
  \begin{displaymath}
    \|\lambda R(\lambda,A)\|
    =\Bigl\|\int_0^\infty \lambda e^{-\lambda t}e^{tA}\,dt\Bigr\|
    \leq C\int_0^\infty \lambda e^{-\lambda t}\,dt
    =C
  \end{displaymath}
  for all $\lambda>0$. In particular, $\lambda^mR(\lambda,A)\to 0$ as
  $\lambda\to 0^+$ for all $m\geq 2$. As $0$ is a pole of the resolvent
  it must therefore be a simple pole. By
  Theorem~\ref{thm:boundary-spectrum}(ii) below this in turn implies
  that all poles of $A$ on the imaginary axis are simple poles of
  $R(\phdot,A)$.

  Next we show that $\spb(A)$ is a dominant spectral value.  Denote by
  $P_{\per}$ the spectral projection corresponding to the peripheral
  spectrum $\sigma_{\per}(A)=\sigma(A)\cap i\mathbb R$; we have
  $P_{\per} \not= 0$ since the peripheral spectrum $\sigma_{\per}(A)$
  contains $\spb(A)$ and is thus non-empty. Let $f \geq 0$ and $t\geq
  0$.  By Proposition~\ref{prop_convergence_on_rangel_of_per_spec_proj},
  there exists a sequence $(t_n) \subseteq[0,\infty)$ with
  $\lim_{n\to\infty} t_n = \infty$ and $e^{t_nA}P_{\per}f \to P_{\per}f$
  for every $f\in E$. Also, $\sigma(A|_{\ker P_{\per}})
  \subseteq\{z\in\bbC\colon\repart z<0\}$, and as shown before,
  $(e^{tA})_{t \geq 0}$ is bounded. Now \cite[Definition~1.1.3 and
  Corollary~5.2.6]{Neerven1996} or \cite[Theorem~2.4]{MR933321} implies
  that $e^{tA}$ converges strongly to $0$ on $\ker(P_{\per})$. As
  $e^{tA}$ is individually eventually positive we conclude that
  \begin{displaymath}
    e^{tA}P_{\per} f
    =P_{\per}e^{tA} f
    = \lim_{n\to\infty} e^{t_nA}P_{\per}e^{tA}f
    = \lim_{n\to\infty} e^{t_nA}e^{tA}f \geq 0
  \end{displaymath}
  for all $t\geq 0$ and $f\geq 0$. In particular, $(e^{tA})_{t\geq 0}$
  restricted to $\im(P_{\per})$ is positive. Setting $t=0$ we also see
  that $P_{\per}\geq 0$ and thus $\im(P_{\per})$ is again a Banach
  lattice when equipped with a suitable equivalent norm as stated in
  Proposition~\ref{prop_range_of_pos_proj}.  Thus, $(e^{tA}|_{\im
    P_{\per}})_{t \geq 0}$ is a bounded, positive $C_0$-semigroup on the
  Banach lattice $\im(P_{\per})$ and the spectral bound of its generator
  is $\spb(A|_{\im P_{\per}}) = 0$. Therefore, the set
  $\sigma_{\per}(A|_{\im P_{\per}}) = \sigma_{\per}(A)$ is imaginary
  additively cyclic, see \cite[Definition~B-III-2.5, Theorem~C-III-2.10
  and Proposition~C-III-2.9]{Arendt1986}. By assumption
  $\sigma_{\per}(A)$ is finite and non-empty, so we conclude that
  $\sigma_{\per}(A) = \{0\}$; in particular $P=P_{\per}$.

  Let us finally show that $P\gg 0$. We have already shown that $P =
  P_{\per}>0$ and that $0$ is a simple pole of the resolvent. Therefore
  $\im(P)=\ker(A)$ and thus $\im(P)$ coincides with the fixed space of
  the operator family $\calT:=(e^{tA})_{t \in [0,\infty)}$.  Hence,
  Lemma~\ref{lem_evtl_pos_family_of_op_eigenvectors}(ii) implies that
  $P\gg 0$.
    
  ``(ii) $\Rightarrow$ (iii)'' Since $P\gg 0$
  Proposition~\ref{prop:projections} shows that $0$ is a simple pole of
  the resolvent and therefore $\im(P)=\ker(A)$. Moreover, as $\spb(A) =
  0$ is a dominant spectral value and $(e^{tA})_{t\geq 0}$ is bounded,
  we conclude from \cite[Corollary~5.2.6]{Neerven1996} or
  \cite[Theorem~2.4]{MR933321} that $e^{tA}\to 0$ strongly on
  $\ker(P)$. Hence, if $f>0$ we have $e^{tA}f = Pf + e^{tA}(I-P)f\to
  Pf\gg 0$ as $t \to \infty$. Therefore, $e^{tA}$ converges strongly to
  the operator $Q:=P\gg 0$ as $t\to\infty$.

  ``(iii) $\Rightarrow$ (i)'' Suppose that $\lim_{t\to\infty}e^{tA}f
  =Qf\gg 0$ for all $f>0$.  As the positive cone has non-empty interior
  in $E_{\mathbb R}$ and as the semigroup is real, we conclude that
  there exists $t_0>0$ such that $e^{tA}f\gg 0$ for all $t>t_0$.

  Hence we have shown the equivalence of (i)--(iii). The proof of the
  implication ``(ii) $\Rightarrow$ (iii)'' shows that $Q=P$ in (iii).
\end{proof}
\begin{remark}
  Some assertions of Theorem~\ref{thm:evpos-semigroups-c_k} have
  counterparts in the theory of \emph{positive} semigroups. For example,
  if $(e^{tA})_{t \ge 0}$ is a positive semigroup and the spectral
  assumptions of Theorem~\ref{thm:evpos-semigroups-c_k} are fulfilled,
  then it follows from \cite[Theorem~C-III-1.1(a) and
  Corollary~C-III-2.12]{Arendt1986} that $\spb(A)$ is a dominant
  spectral value. If the positive semigroup $(e^{tA})_{t \ge 0}$ is
  irreducible and the assumptions of
  Theorem~\ref{thm:evpos-semigroups-c_k} are fulfilled, then it is also
  known (see \cite[Proposition~C-III-3.5]{Arendt1986}) that the spectral
  projection corresponding to $\spb(A)$ is strongly positive and that
  the corresponding eigenspaces of $A$ and $A'$ have the properties that
  were proved in a more general situation in
  Proposition~\ref{prop:projections}.
  
  It is also a classical idea in the theory of positive semigroups that,
  under appropriate assumptions, positivity implies convergence of the
  semigroup, see e.g.~\cite[Section~C-IV-2]{Arendt1986}.
  
  The converse implications ``(ii), (iii) $\Rightarrow$ (i)'' however
  have no counterparts for positive semigroups; they show that
  \emph{eventual} positivity provides the right setting to give
  \emph{characterisations} of Perron--Frobenius type properties and of
  convergence to positive limit operators. In finite dimensions, this
  has already been demonstrated by similar results; see for example
  \cite[Theorem~3.3]{Noutsos2008}. We also refer to our discussion of
  the finite-dimensional case in
  Section~\ref{subsection:appl-finite-dim}.
\end{remark}

Under an additional regularity assumption on the semigroup the
boundedness condition in Theorem~\ref{thm:evpos-semigroups-c_k}(ii) can
be removed, as the following corollary shows. In particular such a
regularity condition is satisfied for analytic semigroups. The corollary
will be useful to check eventual positivity in several applications in
Section~\ref{section_applications_on_c_k}.

\begin{corollary}
  \label{cor:evpos-semigroups-c_k}
  Suppose that $(e^{tA})_{t \geq 0}$ is a real $C_0$-semigroup on
  $E=C(K)$ with $\sigma(A)\neq\emptyset$, and that the peripheral
  spectrum given by \eqref{eq:sigma-per} is finite and consists of poles
  of the resolvent. If the semigroup $(e^{tA})_{t \geq 0}$ is eventually
  norm continuous, then the following assertions are equivalent:
  \begin{enumerate}[\upshape (i)]
  \item The semigroup $(e^{tA})_{t\geq 0}$ is individually eventually
    strongly positive.
  \item $\spb(A)$ is a dominant spectral value of $A$ and the associated
    spectral projections $P$ fulfils $P \gg 0$.
  \end{enumerate}
\end{corollary}
\begin{proof}
  We may assume that $\spb(A) = 0$ and we note that all assumptions of
  Theorem~\ref{thm:evpos-semigroups-c_k} are fulfilled.
  
  Clearly, (i) implies (ii) by Theorem~\ref{thm:evpos-semigroups-c_k}.
  If (ii) holds, then $\spb(A) = 0$ is a simple pole of $R(\phdot,A)$ by
  Proposition~\ref{prop:projections}, so the semigroup $(e^{tA})_{t\geq
    0}$ is bounded on $\im P$. As the semigroup is eventually norm
  continuous, the set $\{\lambda\in\sigma(A)\colon
  \alpha\leq\repart\lambda\}$ is bounded for every $\alpha\in \bbR$ (see
  \cite[Theorem~II.4.18]{Engel2000}) and we conclude that $\spb(A|_{\ker
    P})<0$. From the eventual norm continuity it now follows that the
  growth bound of $(e^{tA}|_{\ker P})_{t \geq 0}$ is negative. Hence
  $(e^{tA})_{t \geq 0}$ is also bounded on $\ker P$.  Therefore,
  condition (ii) of Theorem~\ref{thm:evpos-semigroups-c_k} is fulfilled,
  and hence (i) follows.
\end{proof}

Since we have now several criteria at hand to check whether a semigroup
is individually eventually strongly positive, it is time to give an
example which shows that it is necessary to distinguish between the
individual and the uniform eventual behaviour of a semigroup.

\begin{example}
  \label{ex_ind_evtl_pos_does_not_imply_unif_evtl_pos}
  Consider the Banach lattice $E = C([-1,1])$. Let $\varphi\colon E \to
  \bbC$ be the continuous linear functional given by $\varphi(f) =
  \int_{-1}^1f(x)\,dx$. We thus have the decomposition
  \begin{displaymath}
    E = \langle \one  \rangle \oplus F 
    \qquad \text{with} \qquad
    F:=\ker\varphi \text{.}
  \end{displaymath}
  By $S$ we denote the reflection operator on $F$, given by $Sf(x) =
  f(-x)$ for all $f \in F$ and all $x \in [-1,1]$. Using $S^2 =  I_F$
  we see that $\sigma(S) = \{1,-1\}$ with corresponding eigenspaces
  given by even and odd functions, respectively.

  Now define a bounded linear operator $A$ on $E$ by
  \begin{displaymath}
    A = 0_{\langle \one  \rangle} \oplus (-2 I_F - S) \text{.}
  \end{displaymath}
  We have $\sigma(A) = \{0,-1,-3\}$ and, using $S^2 =  I_F$, we can
  immediately check that
  \begin{align}
    e^{tA} &= I_{\langle\one\rangle}\oplus e^{-2t} \bigl(\cosh(t) I_F-
    \sinh (t) \, S \bigr) \quad\text{and}
    \label{form_sg_counter_ex_reflection_on_c_k_space} \\
    R(\lambda,A) & =\frac{1}{\lambda} I_{\langle\one\rangle}
    \oplus\frac{1}{(\lambda+2)^2 - 1}\bigl((\lambda+2) I_F-S\bigr)
    \text{.}
    \label{form_res_counter_ex_reflection_on_c_k_space}
  \end{align}
  for all $t \geq 0$ and all $\lambda \in \rho(A) = \bbC \setminus
  \{0,-1,-3\}$.  The spectral bound $\spb(A) = 0$ is a dominant spectral
  value and the associated spectral projection $P$ is given by $P f =
  \frac{1}{2}\varphi(f) \one $ and thus strongly positive. Hence, our 
  semigroup is individually eventually strongly positive due to 
  Corollary~\ref{cor:evpos-semigroups-c_k}, and so is the resolvent at
  $\spb(A)$ due to Theorem~\ref{thm:evpos-resolvents-c_k}.

  Now for each $\varepsilon > 0$ choose a function $f_\varepsilon \in
  E_+$ with $\|f_\varepsilon\|_\infty=1$, $\varphi(f_\varepsilon) = \varepsilon$,
  $f_\varepsilon(1) = 1$ and $f_\varepsilon(-1) = 0$. Then
  \begin{displaymath}
    P f_\varepsilon = \frac{\varepsilon}{2} \one  \quad \text{and}
    \quad ( I_E - P) f_\varepsilon = f_\varepsilon -
    \frac{\varepsilon}{2} \one  \text{.}
  \end{displaymath}
  By \eqref{form_sg_counter_ex_reflection_on_c_k_space} we obtain for $t
  \geq 0$ that
  \begin{displaymath}
    e^{tA}f_\varepsilon(-1) = \frac{\varepsilon}{2}\big( 1 -
    e^{-2t}\cosh t + e^{-2t} \sinh t \big) - e^{-2t} \sinh t \text{.}
  \end{displaymath}
  Thus, for each $t \geq 0$, we can choose $\varepsilon > 0$ small 
  enough such that $e^{-tA}f_\varepsilon \not \geq 0$. Therefore
  $(e^{tA})_{t\geq 0}$ is not uniformly eventually positive. In
  particular, it is not uniformly eventually strongly positive. In a
  similar way one can check that the resolvent $R(\phdot, A)$ is not
  uniformly eventually positive at $\spb(A)$.
\end{example}

Noting that the generator of the semigroup in the previous example is
merely bounded, it is natural to ask whether the situation changes if we impose
additional compactness conditions on our semigroup. We proceed with a
further example which is rather disillusioning. We construct an analytic
semigroup with compact resolvent such that the semigroup is individually
eventually strongly positive, but again not even uniformly eventually
positive. The basic idea of the construction is rather similar to
Example~\ref{ex_ind_evtl_pos_does_not_imply_unif_evtl_pos}, but it is
somewhat more technical.
\begin{example}
  \label{ex_ind_evtl_pos_does_not_imply_unif_evtl_pos_for_compact_res}
  Let $c(\bbZ)$ be the subspace of $\ell^\infty(\bbZ)$ given by
  \begin{displaymath}
    c(\bbZ):=\langle\one\rangle\oplus c_0(\bbZ),
  \end{displaymath}
  where $c_0(\bbZ)$ is as usual the set of sequences $(x_n) \in
  \ell^\infty(\bbZ)$ with $x_n \to 0$ as $n \to \pm \infty$.  It is easy
  to see that $c(\bbZ)\simeq C(K)$ for some compact Hausdorff space
  $K$. We can write $c_0(\bbZ)$ as a direct sum of the subspaces of
  symmetric and anti-symmetric sequences, that is, the spaces
  \begin{displaymath}
    \begin{aligned}
      c_0^s&:=\{(x_n)\in c_0(\bbZ)
      \colon\text{$x_n=x_{-n}$ for all $n\in\bbN$}\}\\
      c_0^a&:=\{(x_n)\in c_0(\bbZ) \colon\text{$x_n=-x_{-n}$ for all
        $n\in\bbN_0$}\}
    \end{aligned}
  \end{displaymath}
  If $x=(x_n)\in c_0(\mathbb Z)$, we define the reflection operator $S$
  by $S(x_n):=(x_{-n})$. Then
  \begin{displaymath}
    x = \frac{1}{2}\bigl(x+Sx\bigr)+\frac{1}{2}\bigl(x-Sx\bigr)
    \in c_0^s\oplus c_0^a
  \end{displaymath}
  is the unique decomposition into symmetric and anti-symmetric parts,
  showing that $c(\bbZ)=\langle\one\rangle\oplus c_0^s\oplus c_0^a$.

  Choose strictly positive symmetric sequences $(\alpha_n)_{n\in\bbZ}$
  and $(\beta_n)_{n\in\bbZ}$ with $e^{-n\beta_n} - e^{-n\alpha_n}<0$ for
  all $n\in\mathbb N$, and such that $\alpha_n,\beta_n \to\infty$ as
  $|n|\to\infty$. Now let $g=(g_n)\in\ell^1(\bbZ)\cap c_0^s$ be such
  that $g_n>0$ for all $n\in\mathbb Z$, $\langle g,\one\rangle=1$, and
  \begin{equation}
    \label{eq:neg-sg}
    2g_n+ e^{-n\beta_n} - e^{-n\alpha_n} < 0
  \end{equation}
  for all $n\in\bbN$ large enough. We define a Banach space isomorphism
  $B\in\mathcal L\bigl(c(\bbZ))$ and its inverse by
  \begin{equation}
    \label{eq:B-iso}
    B(c\one+x):=\bigl(c+\langle g,x\rangle\bigr)\one+x
    \quad\text{and}\quad
    B^{-1}(c\one+x):=\bigl(c-\langle g,x\rangle\bigr)\one+x.
  \end{equation}
  for all $c\in\mathbb R$ and $x\in c_0(\bbZ)$. Define the
  multiplication operators $M_\alpha$ and $M_\beta$ on $c_0^s$ and
  $c_0^a$ by $M_\alpha x:=(\alpha_nx_n)$ and $M_\beta x:=(\beta_nx_n)$
  with domains
  \begin{displaymath}
    D(M_\beta) := \{x \in c_0^s\colon\beta x \in c_0^s\}
    \quad\text{and}\quad
    D(M_\alpha) := \{x \in c_0^a\colon\alpha x \in c_0^a\}
  \end{displaymath}
  respectively. Then $-M_\beta$ and $-M_\alpha$ generate bounded
  strongly continuous analytic semigroups on $c_0^s$ and $c_0^a$
  respectively, see \cite[Section~A-I.2.3]{Arendt1986}. We define a
  semigroup $(e^{tA})_{t \ge 0}$ on $c(\bbZ)$ by using the commutative
  diagram in \eqref{eq:cd-sg}.
  \begin{equation}
    \label{eq:cd-sg}
    \begin{CD}
      c(\bbZ) @>{\qquad e^{tA}\qquad}>> c(\bbZ)\\
      @VV{B}V @AA{B^{-1}}A\\
      \langle\one\rangle\oplus c_0^s\oplus c_0^a @>{ I\oplus
        e^{-tM_\beta}\oplus e^{-tM_\alpha}}>> \langle\one\rangle\oplus
      c_0^s\oplus c_0^a
    \end{CD}
  \end{equation}
  The generator of that semigroup is given by
  $A=-B^{-1}(0_{\langle\one\rangle}\oplus M_\beta\oplus
  M_\alpha)B$. Clearly, the operator $A$ has compact resolvent and the
  semigroup $(e^{tA})_{t \ge 0}$ on $c(\bbZ)$ is real, analytic and
  bounded. Moreover, $0$ is an algebraically simple, isolated and dominant 
  eigenvalue of $A$ with eigenvector $\one \gg 0$ and a short computation shows
  that the associated spectral projection $P$ is given by $P(c\one + x)
  = (c + \langle g, x \rangle)\one$ for $c \in \mathbb C$, $x \in
  c_0(\mathbb Z)$. Now, if $c\one + x > 0$, then we can find an element
  $0 < y \in c_0(\bbZ)$ such that even $c\one + x - y \ge 0$.  Hence,
  \begin{displaymath}
    0 < \langle g, y \rangle \le \langle g, c\one + x \rangle =
    c + \langle g, x \rangle \text{,}
  \end{displaymath}  
  which shows that $P$ is strongly positive. Hence, we can apply
  Theorem~\ref{thm:evpos-semigroups-c_k} to conclude that 
  $(e^{tA})_{t \ge 0}$ is individually eventually strongly positive. 
  
  We now show that it is not
  uniformly eventually strongly positive. In fact it is not even
  uniformly eventually positive. Indeed, if $t_0 \geq 0$, then we may
  choose $n\in \bbN$, $n \geq t_0$ such that \eqref{eq:neg-sg} is
  fulfilled for this $n$. We now compute $e^{tA}x$ for $x\in
  c_0(\bbZ)$. From the definitions \eqref{eq:B-iso} of $B$ and $B^{-1}$
  we conclude that
  \begin{equation}
    \label{eq:sg}
    \begin{split}
      e^{tA}x &=B^{-1}\bigl( I_{\langle\one\rangle}
      \oplus e^{-tM_\beta}\oplus e^{-tM_\alpha}\bigl)Bx\\
      &=B^{-1}\Bigl(\langle g,x\rangle\one
      +\frac{1}{2}e^{-tM_\beta}(x+Sx)+\frac{1}{2}e^{-tM_\alpha}(x-Sx)\Bigr)\\
      &=\frac{1}{2}\Bigl(\bigl\langle g,2x
      -e^{-tM_\beta}(x+Sx)\bigr\rangle\one
      +e^{-tM_\beta}(x+Sx)+e^{-tM_\alpha}(x-Sx)\Bigr)
    \end{split}
  \end{equation}
  In particular, if $x>0$, then
  \begin{equation*}
    e^{tA}x
    <\frac{1}{2}\Bigl(2\langle g,x\rangle\one
    +e^{-tM_\beta}(x+Sx)+e^{-tM_\alpha}(x-Sx)\Bigr)
  \end{equation*}
  for all $t>0$. Taking $x=\one_{\{n\}}$ we obtain for the $(-n)$-th
  component of $e^{tA}\one_{\{n\}}$ that
  \begin{displaymath}
    \bigl(e^{tA}\one _{\{n\}}\bigr)_{-n}
    \leq \frac{1}{2}\bigl(2g_n+e^{-t\beta_n}-e^{-t\alpha_n}\bigr)\text{,}
  \end{displaymath}
  and the last term is negative for $t=n$ due to \eqref{eq:neg-sg}.
  Thus, $(e^{tA})_{t \geq 0}$ is not uniformly eventually positive.
\end{example}
\begin{remark}
  \label{rem:uniform-evpos}
  Let $(e^{tA})_{t \ge 0}$ be a real $C_0$-semigroup on $C(K)$ and
  suppose that $\spb(A) = 0$ is a dominant spectral value and a first
  order pole of the resolvent with associated spectral projection
  $P$. It does not seem to be easy to find a simple criterion that
  guarantees the uniform (strong) eventual positivity of
  $(e^{tA})_{t \ge 0}$. To provide the reader with a feeling for the
  situation we point out that a number of candidate criteria which
  appear natural at first glance do not work:
  \begin{itemize}
  \item[(a)] For example, it seems intuitive to require that $P \gg 0$
    and that $e^{tA}$ be uniformly exponentially stable on $\ker
    P$. However, this does not imply uniform eventual positivity as
    Example~\ref{ex_ind_evtl_pos_does_not_imply_unif_evtl_pos}
    shows.
  \item[(b)] Suppose that $(e^{tA})_{t \ge 0}$ is uniformly
    exponentially stable on $\ker P$. If the eigenvalue $\spb(A)$ is
    algebraically simple and if the subspace $\im A = \ker P$ has
    strictly positive distance to the positive normalised functions,
    then it is indeed possible to show that $(e^{tA})_{t \ge 0}$ is
    uniformly eventually positive. However, the reader should be warned
    that this seemingly nice criterion can in fact never be applied in 
    infinite dimensions, for the following reason: if there exists a closed
    subspace $F \subseteq C(K)$ of co-dimension $1$ such that
    \begin{displaymath}
      \inf\{\|v-u\|\colon u\in F,v\in E_+, \|v\|=1\}>0,
    \end{displaymath}
    then one can show that $K$ must actually be finite. We omit the 
    elementary proof.
  \end{itemize}
\end{remark}

\section{Applications on $C(K)$}
\label{section_applications_on_c_k}
We proceed with several applications of the results presented in
Sections~\ref{sec:projections}, \ref{sec:resolvents} and
\ref{sec:semigroups}. We begin with a short treatment of the
finite-dimensional case, where we obtain several results, including a
slight strengthening of known results, as corollaries of the general
theory on $C(K)$. Then we give an application to the
Dirichlet-to-Neumann operator, which was a major motivation for the
development of the theory presented so far. Afterwards, we show that
squares of certain generators on $C(\overline{\Omega})$ generate
eventually positive semigroups, and finally, we present an example of a
delay differential equation whose solution semigroup is eventually
positive but not positive.

\subsection{The finite-dimensional case.}
\label{subsection:appl-finite-dim}
The space $\bbC^n$ with the supremum norm $\|{\phdot}\|_\infty$ is a
complex Banach lattice when its real part $\bbR^n$ is endowed with the
canonical order. Then $(\bbC^n, \|\phdot\|_\infty) =
(C(K),\|\phdot\|_\infty)$, where $K := \{1,\dots,n\}$ is equipped with
the discrete topology, so we can apply our theory.

As noted in the introduction, a sophisticated finite-dimensional theory
of eventually positive operators and semigroups has been developed
during the last twenty years; see for instance
\cite{MR2117663,Noutsos2006} for results about eventually positive
matrices and \cite{Noutsos2008} for eventually positive matrix
semigroups. Also note that somewhat earlier several results for matrices
which posses \emph{some} positive powers were obtained, see
e.g.~\cite{Brauer1961} or \cite[p.~48--54]{Seneta1981}.

In this subsection we illustrate how the results from
Sections~\ref{sec:projections}--\ref{sec:semigroups} imply results from
\cite[Theorem~3.3]{Noutsos2008} as a special case.  The reader should
however be aware that our terminology differs in some points: for
example, matrices and vectors we call ``strongly positive'' are simply
called ``positive'' in \cite{Noutsos2008}, and what we call ``positive''
is ``non-negative'' in \cite{Noutsos2008}.

Note also that since it is easy to see that uniform and individual
eventual (strong) positivity coincide in the finite-dimensional setting,
we will omit the adjectives ``uniform'' and ``individual'' in this
subsection.
\begin{theorem}
  \label{theorem_characterisation_of_evtl_strongly_positive_sg_in_finite_dim}
  For $A \in \bbR^{n \times n}$, the following assertions are
  equivalent:
  \begin{enumerate}[\upshape (i)]
  \item The semigroup $(e^{tA})_{t \geq 0}$ is eventually strongly
    positive.
  \item The spectral bound $\spb(A)$ is a dominant and geometrically
    simple eigenvalue of $A$ and the eigenspaces $\ker(\spb(A) I - A)$ and
    $\ker(\spb(A) I - A^T)$ contain a strongly positive vector.
  \item There exists $c\in\mathbb R$ such that $A+c I$ is eventually
    strongly positive, that is, there exists $k_0>0$ such that
    $(A+c I)^k\gg 0$ for all $k\geq k_0$.
  \end{enumerate}
  If assertions {\upshape(i)}--{\upshape(iii)} are fulfilled, then
  $\spb(A)$ is an algebraically simple eigenvalue of $A$.
\end{theorem}
\begin{proof}
  ```(i) $\Rightarrow$ (ii)'' If (i) holds, then
  Theorem~\ref{thm:evpos-semigroups-c_k} implies that $\spb(A)$ is a
  dominant eigenvalue and that the corresponding spectral projection $P$
  is strongly positive. Hence (ii) follows from
  Proposition~\ref{prop:projections}.

  ``(ii) $\Rightarrow$ (iii)'' According to
  Proposition~\ref{prop:projections}, $\spb(A)$ is an algebraically
  simple eigenvalue with spectral projection $P\gg 0$.  The matrix $A$
  has finitely many eigenvalues $\lambda_k$, $k=1,\dots,m$, other than
  $\spb(A)$. As $\spb(A)$ is dominant, there exists $c\geq 0$ such that
  $\{\lambda_1,\dots,\lambda_m\}$ is contained in the open ball of
  radius $\spb(A)+c > 0$ around $-c$. Hence, $\spb(A)+c$ is the only
  eigenvalue of $A+c I$ of modulus $r:=\spr(A+c I)=\spb(A)+c$ and so
  $\spr\bigl((A+c I)( I-P)\bigr)<r$. It follows that
  \begin{displaymath}
    \lim_{k\to\infty}\bigl[r^{-1}(A+c I)\bigr]^k
    =P+\lim_{k\to\infty}\bigl[r^{-1}(A+c I)(I-P)\bigr]^k
    =P\gg 0.
  \end{displaymath}
  This is the well-known \emph{power method} for computing the dominant
  eigenvalue, see for instance \cite[Theorem~8.2.8]{MR832183}. As $P\gg
  0$ we conclude that $(A+c I)^k\gg 0$ for $k$ large enough.

  ``(iii) $\Rightarrow$ (i)'' We essentially follow the proof from
  \cite[Theorem~3.3]{Noutsos2008}. Set $B:=A+c I$ and assume that
  $B^k\gg 0$ for all $k\geq k_0$. Then
  \begin{displaymath}
    e^{tB}
    =\sum_{k=0}^{k_0-1}\frac{t^k}{k!}B^k
    +\sum_{k=k_0}^{\infty}\frac{t^k}{k!}B^k
    \gg t^{k_0}\Bigl(\frac{1}{k_0!}B^{k_0}
    +\sum_{k=0}^{k_0-1}\frac{t^{k-k_0}}{k!}B^k\Bigr)
  \end{displaymath}
  As $B^{k_0}\gg 0$ and
  \begin{displaymath}
    \lim_{t\to\infty}\sum_{k=0}^{k_0-1}\frac{t^{k-k_0}}{k!}B^k=0
  \end{displaymath}
  there exists $t_0>0$ such that $e^{tB}\gg 0$ for all $t>t_0$. Now (i)
  follows since $e^{tA}=e^{-ct}e^{tB}$.
  
  Finally, (ii) and Proposition~\ref{prop:projections} imply that
  $\spb(A)$ is an algebraically simple eigenvalue of $A$.
\end{proof}

Note that the other conditions in Proposition~\ref{prop:projections} and
Theorem~\ref{thm:evpos-resolvents-c_k} together with
Corollary~\ref{cor:evpos-semigroups-c_k} yield further characterisations
of the eventual positivity of a matrix semigroup. However, since those
assertions are not simplified in the matrix case, we see no reason to
restate them explicitly here.

Let us briefly compare
Theorem~\ref{theorem_characterisation_of_evtl_strongly_positive_sg_in_finite_dim}
with \cite[Theorem~3.3]{Noutsos2008}. Conditions (i) and (iii) in our
Theorem above appear also in \cite[Theorem~3.3]{Noutsos2008} as
conditions (iv) and (ii); our condition (ii) is very similar to
condition (i) there. The latter condition is formulated in terms of the
spectral radius and can easily be rewritten into our condition on the
spectral bound, except for one small difference: the condition in
\cite{Noutsos2008} assumes the spectral radius to be an algebraically
simple eigenvalue whereas we only assume the spectral bound to be
geometrically simple and then \emph{deduce} the algebraic simplicity.
Besides this difference in the assertion of the theorems, we note that
many arguments in \cite{Noutsos2008} are based on the fact that $A^k\gg 0$
for all $k$ large enough. Our proof of the
implication ``(i) $\Rightarrow$ (ii)'' in
Theorem~\ref{theorem_characterisation_of_evtl_strongly_positive_sg_in_finite_dim} 
is new, being based on the characterisations of the spectral projection developed in 
Sections~\ref{sec:projections}--\ref{sec:semigroups}, which have the 
advantage of being applicable in the case of unbounded operators in 
infinite dimensions.

We have seen in Remark \ref{rem:sg-rot}(a) that there are examples of
non-positive, eventually positive semigroups in three (and hence all
higher) dimensions. On the other hand, a one-dimensional real semigroup
is clearly always positive. We now show that in two dimensions, eventual
positivity implies positivity.

\begin{proposition}
  Let $A \in \bbR^{2 \times 2}$. If $(e^{tA})_{t \geq 0}$ is eventually
  (strongly) positive, then $e^{tA}$ is (strongly) positive for each
  $t>0$.
\end{proposition}
\begin{proof}
  As usual we assume that $\spb(A)=0$.  First suppose that $(e^{tA})_{t
    \ge 0}$ is eventually positive. Then $\lambda_1:=\spb(A)\in\mathbb
  R$ is an eigenvalue of $A$ as shown in
  Theorem~\ref{thm:spb-in-spectrum} below. Hence $A$ has two real
  eigenvalues. If $\lambda_1$ has multiplicity two, then either $A=0$
  and $e^{tA}= I\geq 0$, or $A$ is nilpotent and $e^{tA}=I+tA$ is
  eventually positive if and only if $A\geq 0$. In either case eventual
  positivity implies positivity.
  
  Now let $-\lambda_2<0$ be the second eigenvalue of $A$ with
  corresponding eigenvector $u_2$. The general solution of $\dot u=Au$
  is given by $u(t)=au_1+bu_2e^{-\lambda_2t}$ for constants
  $a,b\in\mathbb R$, where $u_1$ is an eigenvector for the eigenvalue
  $\spb(A) = 0$. If $u(0)=au_1+bu_2\geq 0$, then eventual positivity
  implies that
  \begin{equation}
    \label{eq:2d-ode}
    \lim_{t\to\infty}u(t)=\lim_{t\to\infty}(au_1+bu_2e^{-\lambda_2t})=au_1\geq 0.
  \end{equation}
  The trajectory for $t\geq 0$ is a line segment connecting $u(0)$ and
  $au_1$ and thus lies in the positive cone. Hence, $(e^{tA})_{t\geq 0}$
  is positive.
  
  Now assume the semigroup is eventually strongly positive. Then the
  spectral projection $P$ associated with $\spb(A) = 0$ is strongly
  positive and $A$ has two distinct eigenvalues; see
  Theorem~\ref{thm:evpos-semigroups-c_k} and
  Proposition~\ref{prop:projections}.  If $u(0) > 0$, then $au_1 = Pu(0)
  \gg 0$. Thus, either $u(0)=u(t)=au_1\gg 0$ for all $t\geq 0$ or the
  open line segment between $u(0)$ and $au_1$ is in the interior of the
  cone. Hence $(e^{tA})_{t\geq 0}$ is strongly positive.
\end{proof}

\subsection{The Dirichlet-to-Neumann operator}
\label{sec:dirichlet-neumann}
Let $\Omega$ be a bounded Lipschitz domain in $\bbR^n$ and
$\lambda\not\in\sigma(-\Delta)$, where $\Delta$ is the Dirichlet
Laplacian on $\Omega$. Given $\varphi\in L^2(\partial\Omega)$, let $u$
denote the unique solution of $\Delta u+\lambda u=0$ in $\Omega$ and
$u=\varphi$ on $\partial\Omega$. The Dirichlet-to-Neumann operator is
defined by
\begin{displaymath}
  D_\lambda\varphi:=\frac{\partial u}{\partial\nu},
\end{displaymath}
where $\nu$ is the outer unit normal to $\partial\Omega$. A proper
construction of $D_\lambda$ as the generator of a strongly continuous
analytic semigroup on $L^2(\partial\Omega)$ can be found in
\cite{MR3146835,arendt:12:fei}. It is shown in \cite{Daners2014} that
the semigroup $e^{-tD_\lambda}$ is not positive, but only eventually
positive for certain ranges of $\lambda>\lambda_1$, where $\lambda_1$ is
the first eigenvalue of the Dirichlet Laplacian $-\Delta$ on the unit
ball $\Omega=B_1(0)$ in $\mathbb R^2$.

Our goal here is to show that this observation continues to hold in
$C(\Gamma)$, where $\Gamma:=\partial B_1(0)$, and can be obtained using
Theorem~\ref{thm:evpos-semigroups-c_k}. However, to do so, we first need
to know that $e^{-tD_\lambda}$ is in fact a $C_0$-semigroup on
$C(\Gamma)$.  This is the subject of the main theorem in
\cite{MR1288366}. However, it appears that the proof given in
\cite{MR1288366} is not valid without restrictions on the zeroth order
term $a_0$ of the operator $\mathcal A$ (the general second order
elliptic operator considered there). The reason is that in the proofs
provided in \cite{MR1288366}, the Dirichlet-to-Neumann operator is the
one associated with the operator $\mathcal A+\alpha^2 I$ for some real
$\alpha$ large enough, not the one associated with $\mathcal A$ as
claimed in \cite[page~236]{MR1288366}. Hence it actually seems to be an
open problem to establish that $D_\lambda$ generates a $C_0$-semigroup
on $C(\partial\Omega)$ whenever $\lambda > \lambda_1$ is not in the
spectrum of the corresponding Dirichlet Laplacian. Note also that the
conclusion on positivity in the main theorem of \cite{MR1288366} is
not true for the whole range of $\lambda\in\mathbb R$, as was pointed
out in \cite[page~237]{Daners2014}.

Here, to have at least one example (that of the disk in $\bbR^2$) valid
for the complete range of admissible $\lambda\in\mathbb R$, we start by
providing a proof of the following theorem.
\begin{theorem}
  \label{thm:DN-on-C}
  Let $\Gamma=\partial B_1(0)$ be the unit circle in $\bbR^2$ and let
  $\lambda\in\bbR\setminus\sigma(-\Delta)$. When restricted to
  $C(\Gamma)$, the family $(e^{-tD_\lambda})_{t \geq 0}$ is a
  $C_0$-semigroup on $C(\Gamma)$.
\end{theorem}
\begin{proof}
  The semigroup can be represented by a convolution kernel
  \begin{displaymath}
    e^{-tD_\lambda}\varphi
    =G_{\lambda,t}*\varphi
    =\int_{-\pi}^\pi G_{\lambda,t}(\phdot -s)\varphi(s)\,ds,
  \end{displaymath}
  where the kernel $G_{\lambda,t}$ is given by the Fourier series
  \begin{displaymath}
    G_{\lambda,t}(\theta)
    =\frac{1}{2\pi}\sum_{k=-\infty}^\infty e^{-t\mu_k(\lambda)}e^{ik\theta}.
  \end{displaymath}
  Here, $\mu_k(\lambda)=\mu_{-k}(\lambda)$ are the eigenvalues of
  $D_\lambda$ with eigenfunctions $e^{\pm ik\theta}$. As
  $\mu_k(\lambda)$ behaves asymptotically like $k$ as $k\to\infty$, the
  Fourier coefficients of $G_{\lambda,t}$ decay exponentially; see
  \cite[Lemma~4.2]{Daners2014}. Hence, as $C(\Gamma)\hookrightarrow
  L^2(\Gamma)$ we conclude that $e^{-tD_\lambda}$ is analytic as a map
  from $(0,\infty)$ into $\calL\bigl(C(\Gamma)\bigr)$. We only need to
  prove the strong continuity at $t=0$. As shown in
  \cite[Proposition~4.6]{Daners2014} we can represent $G_{\lambda,t}$ in
  terms of the Fej\'er kernels $K_n\geq 0$ in the form
  \begin{displaymath}
    G_{\lambda,t}(\theta)=\sum_{n=1}^\infty nb_n(\lambda,t)K_{n-1}(\theta)
  \end{displaymath}
  with $b_n(\lambda,t) :=e^{-t\mu_{n+1}(\lambda)}
  +e^{-t\mu_{n-1}(\lambda)} -2e^{-t\mu_n(\lambda)}$. As shown in
  \cite[Proposition~4.6]{Daners2014}, for fixed $\lambda$, there exists
  $n_0\geq 1$ such that $b_n(\lambda,t)\geq 0$ for all $n\geq n_0$ and
  all $t>0$. An elementary but not entirely trivial argument now yields
  \begin{align}
    M:=\sup_{t\in(0,1]}\int_{-\pi}^\pi |G_{\lambda,t}(s)|\,ds&<\infty,
    \label{eq:kernel-1}\\
    \lim_{t\to 0}\int_{\alpha}^{2\pi-\alpha} |G_{\lambda,t}(s)|\,ds&=0
    &&\text{for all $\alpha\in(0,\pi)$.}\label{eq:kernel-2}
  \end{align}
  Using these properties we now show that for every $\varphi\in
  C(\Gamma)$ the family $e^{-tD_\lambda}\varphi=G_{\lambda,t}*\varphi$,
  $t\in(0,1]$ is bounded and equicontinuous and therefore relatively
  compact in $C(\Gamma)$ by the Arzel\`a--Ascoli theorem. First, we
  obtain from \eqref{eq:kernel-1} that
  \begin{displaymath}
    \sup_{t\in (0,1]}|u(t)|
    =\Bigl|\int_{-\pi}^\pi G_{\lambda,t}(\theta-s)\varphi(s)\,ds\Bigr|
    \leq M\|\varphi\|_\infty,
  \end{displaymath}
  so the family is bounded. As $u(t):=e^{-tD_\lambda}\varphi\to\varphi$
  in $L^2(\Gamma)$ as $t\to 0$ this implies that we also have
  convergence in $C(\Gamma)$. Indeed, for fixed $\alpha\in(0,\pi)$ we
  have
  \begin{displaymath}
    \begin{split}
      |u(t,\theta+\eta)&-u(t,\theta)|\\
      &=\Bigl|\int_{-\pi}^\pi G_{\lambda,t}(\theta+\eta-s)\varphi(s)\,ds
      -\int_{-\pi}^\pi G_{\lambda,t}(\theta-s)\varphi(s)\,ds\Bigr|\\
      &\leq\int_{-\pi}^\pi |G_{\lambda,t}(\theta-s)|
      |\varphi(s-\eta)-\varphi(s)|\,ds\\
      &\leq 4\int_{\alpha}^{\pi} |G_{\lambda,t}(s)|\,ds\|\varphi\|_\infty\\
      &\qquad\qquad+\int_{-\alpha}^\alpha |G_{\lambda,t}(s)|\,ds
      \sup_{s\in[-\pi,\pi]}|\varphi(s-\eta)-\varphi(s)|\\
      &\leq 4\int_{\alpha}^{\pi}
      |G_{\lambda,t}(s)|\,ds\|\varphi\|_\infty
      +M\sup_{s\in[-\pi,\pi]}|\varphi(s-\eta)-\varphi(s)|
    \end{split}
  \end{displaymath}
  for all $t\in(0,1]$, where we used \eqref{eq:kernel-1} in the last
  inequality.  Fix $\varepsilon>0$. Due to the uniform continuity of
  $\varphi$ on the compact set $\Gamma$ we can choose $\delta>0$ such
  that
  \begin{displaymath}
    M\sup_{s\in[-\pi,\pi]}|\varphi(s-\eta)-\varphi(s)|
    <\frac{\varepsilon}{2}
  \end{displaymath}
  whenever $|\eta|<\delta$. By \eqref{eq:kernel-2} there exists $t_0>0$
  such that
  \begin{displaymath}
    4\int_{\alpha}^{\pi} |G_{\lambda,t}(s)|\,ds\|\varphi\|_\infty
    <\frac{\varepsilon}{2}
  \end{displaymath}
  for all $t\in(0,t_0]$. Hence,
  \begin{equation}
    \label{eq:equicont}
    |u(t,\theta+\eta)-u(t,\theta)|
    <\frac{\varepsilon}{2}+\frac{\varepsilon}{2}
    =\varepsilon 
  \end{equation}
  whenever $|\eta|<\delta$ and $t\in(0,t_0]$. As $u\in
  C([t_0,1],C(\Gamma))$ there exists a possibly smaller $\delta>0$ such
  that \eqref{eq:equicont} holds whenever $|\eta|<\delta$ and
  $t\in(0,1]$. Hence, $u(t)\to u(0)$ in $C(\Gamma)$, showing that
  $e^{-tD_\lambda}$ is a strongly continuous semigroup on $C(\Gamma)$.
\end{proof}
We finally consider the positivity properties of the semigroup.
Regarding eventual positivity, Theorem~\ref{thm:evpos-semigroups-c_k}
allows us to prove the following proposition. In fact, if one considers
the Fourier series representation of the semigroup, one can actually
show that it is \emph{uniformly} eventually strongly positive.
\begin{proposition}
  \label{prop:DN-evpos}
  There exists $\lambda^* \in (\lambda_3,\lambda_4)$ such that
  $e^{-tD_\lambda}$ is individually eventually strongly positive but not
  positive on $C(\Gamma)$ for all $\lambda \in (\lambda_3,\lambda^*)$.
\end{proposition}
There are in fact infinitely many small intervals in which this
holds. We merely discuss one in detail, as an illustration of the
principle.
\begin{proof}
  Recall that $\lambda\in\sigma(-\Delta)$ if and only if
  $J_k(\sqrt{\lambda})=0$ for some Bessel function $J_k$, $k\in\mathbb
  N$. It is shown in \cite{Daners2014} that the eigenvalues of
  $D_\lambda$ are of the form
  \begin{equation}
    \label{eq:DN-eigenvalues}
    \mu_k(\lambda)=\frac{\sqrt{\lambda}J_k'(\sqrt{\lambda})}
    {J_k(\sqrt{\lambda})}
  \end{equation}
  with the corresponding eigenspaces spanned by $\one$ if $k=0$ and by
  $\cos(kt)$, $\sin(kt)$ if $k\geq 1$. A plot of the first few
  eigenvalues as a function of $\lambda$ is shown in
  \cite[Fig.~3]{Daners2014}. The curves have vertical asymptotes at the
  strictly ordered eigenvalues $\lambda_1<\lambda_2<\lambda_3<\dots$ of
  the negative Dirichlet Laplacian on the unit disc $B_1(0)$. The only
  eigenvalue having a strictly positive eigenvector, namely $\one$, is
  $\mu_0(\lambda)$. The corresponding projection is given by
  \begin{displaymath}
    P\varphi=\int_{-\pi}^\pi\varphi(\theta)\,d\theta\one
  \end{displaymath}
  and hence $P\gg 0$. However, note that $\mu_0(\lambda)$ is not always
  the dominant eigenvalue.

  From the explicit values for $\mu_k(\lambda)$ given in
  \eqref{eq:DN-eigenvalues} we can see that $\mu_0(\lambda)$ is dominant
  if $\lambda\in(\lambda_3,\lambda_4)$. This can also clearly be seen
  from \cite[Fig.~3]{Daners2014}, where $\mu_0(\lambda)$ is represented
  by a solid line. Hence, we conclude from
  Theorem~\ref{thm:evpos-semigroups-c_k} that $e^{-tD_\lambda}$ is
  individually eventually positive for
  $\lambda\in(\lambda_3,\lambda_4)$. It is shown in \cite{Daners2014}
  that $(e^{-tD_\lambda})_{t\geq 0}$ is a positive semigroup for
  $\lambda$ close enough to $\lambda_4$. We now show that it is not
  positive if $\lambda$ is in a right neighbourhood of $\lambda_3$. To
  do so we take as an initial condition the Fej\'er kernel
  \begin{equation}
    \label{eq:K3}
    u_0(\theta):=2K_3(\theta)=
    2+3\cos\theta+2\cos2\theta+\cos3\theta
    =\frac{1}{2}\left(\frac{\sin(2\theta)}{\sin(\theta/2)}\right)^2;
  \end{equation}
  see \cite[p.~12]{katznelson:04:iha}. Let
  $u_\lambda(t):=e^{-tD_\lambda}u_0$. We show that $u_\lambda(t)$ is not
  positive for $t$ sufficiently small if $\lambda$ is in a right
  neighbourhood of $\lambda_3$. We do that by showing that $u_\lambda$
  has a negative time derivative at a point where the initial condition
  is zero. Using the formula from \cite[Proposition~4.3(ii)]{Daners2014}
  we see that
  \begin{displaymath}
    \begin{split}
      \dot u_\lambda(\phdot,0)
      &=\frac{d}{dt}e^{-tD_\lambda}u_0\Bigr|_{t=0}
      =-D_\lambda u_0\\
      &=-2\mu_0(\lambda)-3\mu_1(\lambda)\cos\theta
      -2\mu_2(\lambda)\cos2\theta-\mu_3(\lambda)\cos3\theta .
    \end{split}
  \end{displaymath}
  Clearly $u_0> 0$, $u_0(\pi)=0$ and
  \begin{equation}
    \label{eq:DN-derivative_at_zero}
    \dot u_\lambda(\pi,0)
    =-2\mu_0(\lambda)+3\mu_1(\lambda)-2\mu_2(\lambda)+\mu_3(\lambda)
  \end{equation}
  for all $\lambda\in(\lambda_3,\lambda_4)$. Further note that
  $J_2(\sqrt{\lambda_3})=0$, so that $\mu_2(\lambda)\to\infty$ as
  $\lambda\downarrow\lambda_3$; see \cite[p.~244]{Daners2014}. As the
  eigenvalues $\mu_0(\lambda),\mu_1(\lambda)$ and $\mu_3(\lambda)$
  remain bounded in a right neighbourhood of $\lambda_3$,
  \eqref{eq:DN-derivative_at_zero} implies that $\dot
  u_\lambda(\pi,0)\to-\infty$ as $\lambda\downarrow\lambda_3$. This can
  be seen in \cite[Fig.~3]{Daners2014}. In particular, because
  $u(\pi,0)=u_0(0)=0$ we conclude that $\dot u_\lambda(\pi,0)<0$ if
  $\lambda$ is in a right neighbourhood of $\lambda_3$. Hence
  $e^{-tD_\lambda}u_0$ is not positive, but only eventually positive.
\end{proof}

\subsection{Squares of Generators}
We saw in Proposition~\ref{prop_square_of_resolvent_pos_operator_c_k}
that the resolvent of the operator $A:=(iB)^2=-B^2$ is eventually
positive at $\lambda_0 < 0$ if $(\lambda_0,0] \subseteq\rho(A)$ and if
$B$ is resolvent positive in $0$.  We now show that such a situation
gives rise to eventually positive semigroups. However, note that even if
$B$ generates a strongly continuous semigroup, that is not automatically
the case for $A=-B^2$. There are special conditions when this is the
case, namely if $B$ generates a group; see
\cite[Theorem~A-II-1.15]{Arendt1986} or
\cite[Corollary~II.4.9]{Engel2000}. We do not wish to assume this, 
but instead work with sectorial operators.

Let us therefore recall some important notions: Let $E$ be a complex
Banach space and $\theta \in (0,\pi]$. By $\Sigma_\theta :=
\{re^{i\varphi}\colon\text{$r > 0$, $\varphi \in (-\theta,\theta)$}\}$
we denote the \emph{open sector of angle $\theta$}. Now, let $\theta \in
(0,\pi/2]$. A $C_0$-semigroup $(e^{tA})_{t \ge 0}$ is called
\emph{analytic of angle $\theta$} if it has an extension $(e^{zA})_{z
  \in \Sigma_\theta \cup \{0\}}$ which is analytic on $\Sigma_{\theta}$
and which is bounded on $\{z \in \Sigma_{\theta'}\colon |z| < 1\}$ for
each $\theta' \in (0,\theta)$.  The semigroup is called \emph{analytic}
if it is analytic of some angle $\theta \in (0,\pi/2]$. The
$C_0$-semigroup $(e^{tA})_{t \ge 0}$ is called \emph{bounded analytic of
  angle $\theta \in (0,\pi/2]$} if it is analytic of angle $\theta$ and
if its extension is bounded on $\Sigma_{\theta'}$ for each $\theta' \in
(0,\theta)$. Finally, an operator $A$ on $E$ is called \emph{sectorial
  of angle $\theta \in (0,\pi/2]$} if $\rho(A) \supset \Sigma_{\pi/2 +
  \theta}$ and if $\sup_{\lambda \in \Sigma_{\pi/2 + \theta'}} \|\lambda
R(\lambda,A)\| <\infty$ for each $\theta' \in (0,\theta)$. Here we use
the definition of sectorial operators in
\cite[Definition~II.4.1]{Engel2000}, which differs from that in other
sources.

Let $\theta \in (0,\pi/2]$.  It is well known that a densely defined
operator $A$ generates a $C_0$-semigroup which is bounded analytic of
angle $\theta$ if and only if $A$ is sectorial of angle $\theta$, see
\cite[Theorem~3.7.11 and Corollary~3.3.11]{Arendt2011}.  For our
subsequent application to the Robin Laplace operator we will need the
following observation:

\begin{lemma}
  \label{lem:boundedness-of-holomorphic-semigroup}
  Let $(e^{tA})_{t \ge 0}$ be an analytic $C_0$-semigroup of angle
  $\theta \in (0,\pi/2]$ such that $\rho(A) \supset
  \Sigma_{\theta+\pi/2}$ and $0 \not\in \sigma(A)$.  Then $(e^{tA})_{t
    \ge 0}$ is bounded analytic of angle $\theta$.
\end{lemma}
\begin{proof}
  Let $\theta' \in (0,\theta)$. Using
  \cite[Proposition~3.7.2(d)]{Arendt2011} we see that
  $(e^{te^{i\theta'}A})_{t \ge 0}$ and $(e^{te^{-i\theta'}A})_{t \ge 0}$
  are analytic $C_0$-semigroups. Moreover, $\spb(e^{i\theta'}A) < 0$ and
  $\spb(e^{-i\theta'}A) < 0$. Hence, both semigroups
  $(e^{te^{i\theta'}A})_{t \ge 0}$ and $(e^{te^{-i\theta'}A})_{t \ge 0}$
  converge to $0$ with respect to the operator norm and are therefore
  bounded in norm by some constant $M \ge 1$.  This in turn implies that
  $(e^{zA})_{z \in \Sigma_{\theta'}}$ is bounded by $M^2$, which shows
  the assertion.
\end{proof}

Let us now prove a result on squares of generators and eventual
positivity.

\begin{proposition}
  \label{prop:square-generator}
  Let $B$ generate an analytic $C_0$-semigroup of angle $\pi/2$ on $E =
  C(K)$, and suppose that $\sigma(B) \subset(-\infty,0)$ is non-empty and
  that $B$ has compact resolvent. If $R(0,B) \gg 0$, then $A := -B^2$
  generates an analytic $C_0$-semigroup on $E$ which is individually
  eventually strongly positive.
\end{proposition}
\begin{proof}
  By Lemma~\ref{lem:boundedness-of-holomorphic-semigroup}, the operator
  $B$ is sectorial of angle $\pi/2$. As is well known (and easy to
  check), this implies that $A$ is sectorial of angle $\pi/2$, too.
  Hence, $A$ generates an analytic $C_0$-semigroup.
  
  Since $\sigma(A) \subseteq(-\infty,0)$ is non-empty, $\spb(A) < 0$ is
  clearly a dominant spectral value of $A$. By assumption $R(0,B) \gg 0$
  and hence an application of
  Proposition~\ref{prop_square_of_resolvent_pos_operator_c_k} shows that
  $R(\phdot,A)$ is uniformly eventually strongly positive at $\spb(A)$.
  As $B$ has compact resolvent, the same is true for $A$ and hence
  $\spb(A)$ is a pole of $R(\phdot,A)$, see
  \cite[Corollary~IV.1.19]{Engel2000}. Now
  Theorem~\ref{thm:evpos-resolvents-c_k} implies that the spectral
  projection of $A$ associated with $\spb(A)$ is strongly positive. Hence,
  Corollary~\ref{cor:evpos-semigroups-c_k} shows that $(e^{tA})_{t \geq
    0}$ is individually eventually strongly positive.
\end{proof}

\subsection{The square of the Robin Laplacian on $C(\overline{\Omega})$.}
We will apply Proposition~\ref{prop:square-generator} to a particular
operator, the Robin Laplacian. To that end, let $\Omega \subseteq\bbR^n$
be a bounded domain of class $C^2$ and let $\beta \in C^1(\partial
\Omega)$ with $\beta\gg 0$. Denote by $\Delta_R^c$ the realisation of
the Laplacian on $C(\overline{\Omega})$ subject to the Robin boundary
condition
\begin{displaymath}
  \frac{\partial}{\partial \nu} u + \beta u = 0 \quad \text{on } 
  \partial \Omega \text{.}
\end{displaymath}
It is shown in \cite[Theorems~8.2 and~6.1]{amann:83:dss} that
$\Delta_R^c$ generates a compact and strongly positive semigroup
on $C(\overline{\Omega})$, which is analytic of angle $\pi/2$ by
\cite[Theorem~3.3]{Warma2006}. Moreover, as $\beta\gg 0$, we certainly
have $\sigma(\Delta_R^c)\subseteq(-\infty,0)$ and clearly,
$\sigma(\Delta_R^c) \neq \emptyset$.
\begin{proposition}
  Under our assumptions on $\beta$ and $\Omega$, the operator $A =
  -(\Delta_R^c)^2$ generates a $C_0$-semigroup on $C(\overline{\Omega})$
  which is individually eventually strongly positive but not positive.
\end{proposition}
\begin{proof}
  From the above discussion and from
  Proposition~\ref{prop:square-generator} it follows that $A$ generates
  an analytic $C_0$-semigroup $(e^{tA})_{t \geq 0}$ which is
  individually eventually strongly positive.  The semigroup $(e^{tA})_{t
    \geq 0}$ is not positive because the restriction of $A$ to
  $C_c^\infty(\Omega)$ is the bi-Laplacian acting on
  $C_c^\infty(\Omega)$, whose extensions cannot generate a positive
  semigroup by \cite[Proposition~2.2]{ABR90}.
\end{proof}

\subsection{A delay differential equation}
\label{subsection:appl-delay-equ}
We consider the time evolution of a complex value $y(t)$, where the rate
of change of $y(t)$ depends on the values of $y$ on the past time
interval $[t-2,t]$, more precisely being given by
\begin{equation}
  y'(t) = \int_{t-2}^{t-1} y(s) \, ds - \int_{t-1}^t y(s) \, ds
  \text{.} \label{form_delay_differential_equation}
\end{equation}
This is called a \emph{delay differential equation} and it can be
analysed by means of evolution semigroups as described in
\cite[Section~IV.2.8]{Engel2000} (with a different time scale) and in
\cite[Section~VI.6]{Engel2000}. Note that the latter section deals with
a more general situation; in their notation, we obtain the setting for
our example by defining $Y := \bbC$ and $B := 0$. We can reformulate
\eqref{form_delay_differential_equation} as the abstract Cauchy problem
$\dot{u}(t,\phdot) = Au(t,\phdot)$ on the space $C([-2,0])$, where the
operator $A$ is given by
\begin{equation}
  \begin{aligned}
    \label{form_generator_of_dealy_semigroup}
    D(A) &:=\Bigl\{f \in C^1([-2,0])\colon f'(0)
    = \int_{-2}^{-1}f(x)\,dx-\int_{-1}^0f(x)\,dx\Bigr\}\text{,} \\
    Af &:= f' \text{.}
  \end{aligned}
\end{equation}
For a derivation of this reformulation, we refer to the references
quoted above. There, it is also shown that the operator $A$ generates a
$C_0$-semigroup on $C([-2,0])$. Our aim here is to prove that this
semigroup is individually eventually strongly positive.

\begin{proposition}
  The operator $A\colon C([-2,0]) \supset D(A) \to C([-2,0])$ given by
  \eqref{form_generator_of_dealy_semigroup} has the following
  properties:
  \begin{enumerate}[\upshape (i)]
  \item The spectral bound $\spb(A)$ equals $0$; moreover, it is a
    dominant spectral value and a pole of the resolvent.
  \item The spectral projection $P$ associated with $\spb(A)$ is strongly
    positive.
  \item The semigroup $(e^{tA})_{t \geq 0}$ on $C([-2,0])$ is
    individually eventually strongly positive.
  \item The semigroup $(e^{tA})_{t \geq 0}$ is not positive.
  \end{enumerate}
\end{proposition}
\begin{proof}
  For the proof we introduce the functional $\Phi\colon C([-2,0]) \to
  \bbC$ given by
  \begin{displaymath}
    \Phi(f) = \int_{-2}^{-1}f(x)\,dx-\int_{-1}^0f(x)\,dx.
  \end{displaymath}

  (i) Since the embedding $D(A) \hookrightarrow C([-2,0])$ is compact
  due to the Arzel\`{a}--Ascoli theorem, $A$ has compact resolvent
  (cf.~also \cite[p.~256]{Engel2000}). Hence, all spectral values are
  poles of the resolvent $R(\phdot,A)$ (see
  \cite[Corollary~IV.1.19]{Engel2000}).

  Let us now show that $0$ is a dominant spectral value of $A$. By
  \cite[formula~(6.11) on p.~427]{Engel2000}, the spectral values of $A$
  are exactly the complex numbers $\lambda$ which fulfil the equation
  $\lambda - \Phi(e^{\lambda \phdot}) = 0$. Using our definition of
  $\Phi$, we obtain after a short computation that
  \begin{displaymath}
    \lambda \in \sigma(A) \quad \Leftrightarrow \quad -\lambda^2 =
    (1 - e^{-\lambda})^2 \text{.}
  \end{displaymath}
  Clearly, $0$ is a solution of this equation, so $0 \in \sigma(A)$. One
  can also see directly that $\one_{[-2,0]}$ is an eigenfunction for
  $0$. To show that there are no other spectral values with non-negative
  real part, note that the above equation is satisfied if and only if one
  of the following two equations is fulfilled:
  \begin{align}
    i \lambda & = 1 - e^{-\lambda} \text{,}
    \label{form_char_equation_for_delay_difff_equ_first_case} \\
    -i \lambda & = 1 - e^{-\lambda} \text{.}
    \label{form_char_equation_for_delay_difff_equ_second_case}
  \end{align}
  Since $\lambda \in \bbC$ fulfils
  \eqref{form_char_equation_for_delay_difff_equ_first_case} if and only
  if $\overline{\lambda}$ fulfils
  \eqref{form_char_equation_for_delay_difff_equ_second_case}, it is
  sufficient to consider the first equation. Writing $\lambda$ as
  $\lambda = \alpha + i \beta$, where $\alpha, \beta \in \bbR$, we
  obtain that \eqref{form_char_equation_for_delay_difff_equ_first_case}
  is equivalent to the system
  \begin{align}
    - \beta & = 1 - e^{-\alpha}\cos \beta \text{,}
    \label{form_char_equation_for_delay_difff_equ_real_part} \\
    \alpha & = e^{-\alpha} \sin \beta \text{.}
    \label{form_char_equation_for_delay_difff_equ_im_part}
  \end{align}
  Suppose that $\alpha \geq 0$. Then
  \eqref{form_char_equation_for_delay_difff_equ_real_part} yields that
  $\beta \in [-2,0]$. Hence, we obtain from
  \eqref{form_char_equation_for_delay_difff_equ_im_part} that $0 \leq
  \alpha = e^{- \alpha} \sin \beta \leq 0$. Thus we conclude that $\sin
  \beta = 0$, so $\beta = 0$ and, finally, $\alpha = 0$. So we have
  shown that the only spectral value of $A$ with non-negative real part
  is given by $\lambda = 0$.
      
  (ii) Clearly, the eigenvalue $\spb(A) = 0$ of $A$ is geometrically
  simple and its eigenspace is spanned by $\one_{[-2,0]}$. Moreover,
  consider the functional $\varphi \in C([-2,0])'$ which is given by
  \begin{displaymath}
    \varphi(f) = f(0) + \int_{-2}^{-1} (2+x)f(x)\, dx  + 
    \int_{-1}^0 -xf(x) \, dx.
  \end{displaymath}
  The functional $\varphi$ is strongly positive and using the definition
  of the adjoint $A'$, it is easy to check that $\varphi$ is an
  eigenvector of $A'$ for the eigenvalue $0$. Hence, we conclude from
  Proposition~\ref{prop:projections} that $P \gg 0$.
  
  Alternatively, we could use the explicit formula for the resolvent of
  $A$ which is given in \cite[Proposition~VI.6.7]{Engel2000} to compute
  that $R(\phdot,A)$ is individually eventually strongly positive at
  $\spb(A) = 0$. Then it follows from
  Theorem~\ref{thm:evpos-resolvents-c_k} that $P \gg 0$.

  (iii) Since $(e^{tA})_{t \geq 0}$ is eventually norm continuous (see
  \cite[Theorem~VI.6.6]{Engel2000}), assertion (iii) follows from (i),
  (ii) and Corollary~\ref{cor:evpos-semigroups-c_k}.

  (iv) This follows from \cite[Example B-II.1.22]{Arendt1986} by
  rescaling the time scale from $[-1,0]$ to $[-2,0]$.
\end{proof}

\section{The spectral bound of eventually positive semigroups}
\label{sec:spectral_bound}

In this section we consider eventual positivity not only on
$C(K)$-spaces but also on arbitrary Banach lattices. By analogy with
Definition~\ref{def_eventual_positivity_for_sg_on_c_k}, we call a
$C_0$-semigroup $(e^{tA})_{t \ge 0}$ on a complex Banach lattice $E$
\emph{individually eventually positive} if for each $0 \le f \in E$
there is a $t_0 \ge 0$ such that $e^{tA}f \ge 0$ for all $t \ge
t_0$. Our aim is to show that such semigroups have many properties which
are already well known for positive semigroups.

We note that some of the results in this section should also hold on
more general ordered spaces than Banach lattices. For example,
Proposition~\ref{prop:laplace-transform} and its corollaries also hold
on ordered Banach spaces with normal cones. One could also try to
consider eventually positive semigroups on operator algebras, as was
done for positive semigroups in \cite[Chapter~IV]{Arendt1986}. However,
we shall not pursue this here.

Recall that if $A$ generates a $C_0$-semigroup $(e^{tA})_{t \ge 0}$,
then $\gbd(A)$ denotes the growth bound of this semigroup.  We start
with the following representation formula for the resolvent: if
$(e^{tA})_{t \geq 0}$ is a $C_0$-semigroup on a Banach space $E$, it is
well known that for $\repart \lambda > \gbd(A)$ the resolvent
$R(\lambda, A)$ can by represented as the Laplace transform of the
semigroup, that is,
\begin{equation}
  \label{eq:laplace-transform} 
  R(\lambda,A)f
  =\int_0^\infty e^{-t\lambda} e^{tA}f\,dt
\end{equation}
for all $f \in E$, where the integral is absolutely convergent. If the
spectral bound and the growth bound $\gbd(A)$ of $A$ do not coincide,
this formula may in general fail for $\spb(A) < \repart \lambda \le
\gbd(A)$; see \cite[Example~5.1.10, Theorem~5.1.9 and the end of
p.\,342]{Arendt2011}.  It is a special feature of positive semigroups
that \eqref{eq:laplace-transform} holds in the strip $\spb(A) < \repart
\lambda \leq \gbd(A)$, where the integral is to be understood as an
improper Riemann integral; see \cite[Theorem~C-III-1.2]{Arendt1986} or
\cite[Theorem~5.3.1 and Proposition~5.1.4]{Arendt2011}. We now show that
this property holds for individually eventually positive semigroups as
well.

\begin{proposition}
  \label{prop:laplace-transform}
  Let $(e^{tA})_{t \geq 0}$ be an individually eventually positive
  $C_0$-semigroup on a complex Banach lattice $E$. Then the Laplace
  transform representation \eqref{eq:laplace-transform} is valid
  whenever $\repart \lambda > \spb(A)$ and $f\in E$, where the integral
  converges as an improper Riemann integral.
\end{proposition}
\begin{proof}
  We may assume that $f \geq 0$. Let $t_0 \geq 0$ such that $e^{tA}f \ge
  0$ for all $t \geq t_0$ and consider the functions $u,v \colon
  [0,\infty) \to E$, $u(t) = e^{tA}f$, $v(t) = e^{(t_0 + t)A}f$. By
  $\abs{u},\abs{v}$ we denote the abscissas of convergence of the
  Laplace transforms $\hat v, \hat u$, as for instance defined in
  \cite[Section~1.4]{Arendt2011}.  Clearly, $\abs(u) \le \gbd(A)$. From
  the formula
  \begin{displaymath}
    \int_0^T e^{-t\lambda} u(t)\,dt
    =\int_0^{t_0} e^{-t\lambda}u(t) \, dt
    +e^{-t_0\lambda}\int_0^{T-t_0} e^{-t\lambda}v(t) \, dt 
  \end{displaymath}
  for all $T>t_0$, we conclude that $\abs(u) = \abs(v)$. Both Laplace
  transforms $\hat u(\lambda)$ and $\hat v(\lambda)$ exist and are
  analytic on the half plane $\repart \lambda > \abs(v)$ (see
  \cite[Proposition~1.4.1 and Theorem~1.5.1]{Arendt2011}). The function
  $\hat u(\lambda)$ coincides with $R(\lambda,A)f$ for $\repart \lambda
  > \gbd(A)$ and so, due to the identity theorem for analytic functions,
  also for $\repart \lambda > \abs(v)$. Hence, we only have to show
  $\abs(v) \le \spb(A)$.

  Assume for a contradiction that $\abs(v) > \spb(A)$.  Then $\abs(v) >
  - \infty$ and $\abs(v) \not\in \sigma(A)$.  Since $v(t) \ge 0$ for all
  $t \ge 0$, \cite[Theorem~1.5.3]{Arendt2011} implies that $\hat v$ has
  a singularity at $\abs(v)$, i.e.~$\hat v$ cannot be analytically
  extended to an open neighbourhood of $\abs(v)$. As we have
  \begin{displaymath}
    R(\lambda,A)f
    =\hat u(\lambda)
    =\int_0^{t_0} e^{-t\lambda}u(t) \, dt+e^{-t_0\lambda} \hat v(\lambda)
  \end{displaymath}
  for $\repart \lambda > \abs(v)$, we conclude that $R(\phdot,A)f$ also
  has a a singularity at $\abs(v)$. This contradicts $\abs(v) \not\in
  \sigma(A)$.
\end{proof}
Note that the proof of the above proposition in fact shows that
\cite[Theorem~1.5.3]{Arendt2011} holds for eventually positive
\emph{functions}.

Proposition~\ref{prop:laplace-transform} yields the following stability
result, which is already known for positive semigroups; see
\cite[Proposition~VI.1.14]{Engel2000}.
\begin{corollary}
  \label{cor:stability-on-domain}
  Let $(e^{tA})_{t \geq 0}$ be an individually eventually positive
  $C_0$-semigroup on a complex Banach lattice $E$. Then $\spb(A)<0$ if
  and only if there exists $\varepsilon > 0$ such that $e^{\varepsilon
    t}\|e^{tA}f\| \to 0$ as $t \to \infty$ for every $f \in D(A)$.
\end{corollary}
\begin{proof}
  If we use the notation from \cite[p.~343]{Arendt2011}, then
  Proposition~\ref{prop:laplace-transform} shows that $\abs(e^{\phdot
    A}) = \spb(A)$ and hence
  \begin{displaymath}
    \spb(A)=\omega_1(e^{\phdot A})
    :=\inf\bigl\{\omega\in\mathbb R\colon\text{$\forall x\in D(A)$
      $\exists M\geq 1$ with $\|e^{tA}x\|\leq Me^{\omega t}$
      $\forall t\geq 0$}\bigr\}
  \end{displaymath}
by
  \cite[Theorem~5.1.9]{Arendt2011}. This implies the assertion.
\end{proof}

Another corollary of Proposition \ref{prop:laplace-transform} is the
following ``asymptotic positivity'' of the resolvent.

\begin{corollary}
  \label{cor:asymp_pos_reslovent}
  Let $(e^{tA})_{t \geq 0}$ be an individually eventually positive
  $C_0$-semigroup on a Banach lattice $E$ with $s(A) > -\infty$.  Then
  for every $f\geq 0$ we have
  \begin{displaymath}
    \dist\bigl((\lambda - \spb(A) )R(\lambda,A)f,E_+ \bigr) \to 0 \quad
    \text{as } \lambda \downarrow \spb(A) \text{.}
  \end{displaymath}
\end{corollary}
\begin{proof}
  We may assume $\spb(A) = 0$. Let $f \geq 0$ and choose $t_0$ such that
  $e^{tA}f \geq 0$ whenever $t \geq t_0$. By Proposition
  \ref{prop:laplace-transform}, we obtain for $\lambda > 0$ that
  \begin{displaymath}
    \dist\bigl(\lambda R(\lambda,A)f,E_+\bigr)
    \leq\Bigl\|\lambda\int_0^{t_0}e^{-t\lambda}e^{t A} f\,dt\Bigr\|
    \leq C(1-e^{-t_0\lambda})
    \text{,}
  \end{displaymath}
  where $C = \sup_{0 \leq t \leq t_0} \|e^{tA}f\|$. The Corollary
  follows by letting $\lambda \downarrow 0$.
\end{proof}

For a positive semigroup $(e^{tA})_{t \geq 0}$ the estimate
$|R(\lambda,A)f| \leq R(\repart \lambda,A)|f|$ holds for all $f \in E$
whenever $\repart \lambda > \spb(A)$; this is an easy consequence of the
validity of formula \eqref{eq:laplace-transform} for $\repart \lambda >
\spb(A)$; see \cite[Corollary~C-III-1.3]{Arendt1986}.  The following
lemma provides us with a slightly weaker result for individually
eventually positive semigroups and for real elements $f \in E_\bbR$.
\begin{lemma}
  \label{lem_modulus_estimate_for_resolvent}
  Let $(e^{tA})_{t \geq 0}$ be an individually eventually positive
  $C_0$-semigroup on a complex Banach lattice $E$. For each $f \in
  E_\bbR$ there is a bounded map $r_f\colon (\spb(A), \infty) \to E$
  which satisfies the following properties:
  \begin{enumerate}[\upshape (i)]
  \item We have
    \begin{math}
      |R(\lambda,A)f| \leq R(\repart \lambda,A)|f| + r_f(\repart\lambda)
    \end{math}
    for all $\repart \lambda > \spb(A)$;
  \item If $\spb(A) > -\infty$, then $r_f$ is norm-bounded on
    $(\spb(A),\infty)$.
  \item If $\spb(A) = -\infty$, then $r_f$ is norm-bounded on
    $(\alpha,\infty)$ for every $\alpha \in \bbR$.
  \end{enumerate}
\end{lemma}
\begin{proof}
  Let $f \in E_\bbR$ and let $t_0 \geq 0$ such that $e^{tA}f^+ \geq 0$
  and $e^{tA}f^- \geq 0$ for all $t \geq t_0$. Then $|e^{tA}f| \leq
  e^{tA}|f|$ for all $t \geq t_0$. For $\repart \lambda > \spb(A)$ and
  $T \ge t_0$ we have
  \begin{multline*}
    \Bigl|\int_0^T e^{-\lambda t}e^{tA}f\,dt\Bigr|
    \leq \int_0^T e^{-t\repart \lambda} |e^{tA} f|\,dt \\
    \leq \int_0^T e^{-\repart \lambda t}e^{tA}|f|\,dt + \int_0^{t_0}
    e^{-t\repart \lambda} \bigl( |e^{tA}f| - e^{tA}|f| \bigr)\,dt
    \text{.}
  \end{multline*}
  Letting $T \to \infty$ we conclude from
  Proposition~\ref{prop:laplace-transform} that
  \begin{displaymath}
    |R(\lambda,A)f|
    \leq R(\repart \lambda, A)|f|+ \int_0^{t_0} e^{-t\repart \lambda}
    \bigl( |e^{tA}f| - e^{tA}|f| \bigr)\,dt \text{.}
  \end{displaymath}
  Defining the last integral as $r_f(\repart \lambda)$, we obtain
  (i)--(iii).
\end{proof}
Recall that for a $C_0$-semigroup $(e^{tA})_{t \ge 0}$ on a complex
Banach space, the quantity
\begin{displaymath}
  \spb_0(A) := \inf\bigl\{\omega > \spb(A)\colon \sup_{\repart \lambda > \omega}
  \|R(\lambda,A)\| < \infty\bigr\}
\end{displaymath}
is called the \emph{abscissa of uniform boundedness of the resolvent} or
the \emph{pseudo-spectral bound} of $A$.

\begin{corollary}
  \label{cor:boundedness-of-resolvent}
  Let $(e^{tA})_{t \geq 0}$ be an individually eventually positive
  semigroups on a complex Banach lattice $E$. Then $\spb(A) =
  \spb_0(A)$.
\end{corollary}
\begin{proof}
  This readily follows from
  Lemma~\ref{lem_modulus_estimate_for_resolvent} and the uniform
  boundedness principle.
\end{proof}
Using Lemma~\ref{lem_modulus_estimate_for_resolvent}, we are able to
prove a generalisation of a well-known result for positive semigroups as
for instance given in \cite[Corollary~C-III.1.4]{Arendt1986} or
\cite[Theorem~5.3.1]{Arendt2011}
\begin{theorem}
  \label{thm:spb-in-spectrum}
  Let $(e^{tA})_{t \geq 0}$ be an individually eventually positive
  $C_0$-semigroup on a complex Banach lattice $E$. If $\sigma(A) \not=
  \emptyset$, then $\spb(A) \in \sigma(A)$.
\end{theorem}
\begin{proof}
  Let $\sigma(A)\neq\emptyset$ and choose a sequence $(\lambda_n)$ with
  $\repart \lambda_n > \spb(A)$ such that $\dist(\lambda_n,\sigma(A))
  \to 0$.  Then $\repart \lambda_n \to \spb(A)$ and $\|R(\lambda_n,A)\|
  \to \infty$. By the uniform boundedness principle, there is an $f \in
  E$ and a subsequence $(\lambda_{n_k})$ of $(\lambda_n)$ such that
  $\|R(\lambda_{n_k},A)f\| \to \infty$. We may in fact choose $f$ to be
  real. Thus, Lemma~\ref{lem_modulus_estimate_for_resolvent} implies
  that $\|R(\repart \lambda_{n_k},A)f\| \to \infty$. As $\repart
  \lambda_{n_k} \to \spb(A)$, we conclude that $\spb(A) \in \sigma(A)$.
\end{proof}
If we know that the spectral bound $\spb(A)$ is a pole of the resolvent,
then we can draw a conclusion on the order of any other pole in the
peripheral spectrum $\sigma_{\per}(A)$, similar to the case of positive
semigroups; see \cite[Corollary C-III-1.5]{Arendt1986}.

\begin{theorem}
  \label{thm:boundary-spectrum}
  Let $(e^{tA})_{t \geq 0}$ be an individually eventually positive
  $C_0$-semigroup on a complex Banach lattice $E$. Suppose that $\spb(A)
  > -\infty$ is a pole of $R(\phdot,A)$ of order $m \in \bbN$. Then we
  have the following assertions.
  \begin{enumerate}[\upshape (i)]
  \item The number $\spb(A)$ is an eigenvalue of $A$ admitting a
    positive eigenvector.
  \item Every pole of $R(\phdot,A)$ in $\sigma_{\per}(A)$ has order at
    most~$m$.
  \end{enumerate}
\end{theorem}
\begin{proof}
  As usual, without loss of generality we assume that $\spb(A)=0$.  To
  prove (i) we use the Laurent expansion \eqref{eq:Laurent-exp} of
  $R(\phdot,A)$ about $\spb(A)=0$. As $0$ is a pole of order $m$ it
  follows that $\lambda^mR(\lambda,A)\to U^{m-1}$ in $\calL(E)$. Also
  recall from Remark~\ref{rem:Laurent-exp} that $\im(U^{m-1}) \not=
  \{0\}$ consists of eigenvectors of $A$ to the eigenvalue $0$. It
  follows from Corollary~\ref{cor:asymp_pos_reslovent} that $U^{m-1}$ is
  positive, and hence $A$ has a positive eigenvector.

  To prove (ii) assume that $\lambda_0\in i\bbR$ is a pole of
  $R(\phdot,A)$. Applying Lemma~\ref{lem_modulus_estimate_for_resolvent}
  we see that for $\lambda=\lambda_0+\alpha$ with $\alpha>0$ we have
  \begin{equation}
    \label{eq:pole-test}
    \bigl|(\lambda - \lambda_0)^kR(\lambda,A)f\bigr|
    \leq \alpha^k R(\alpha,A)|f|+\alpha^k r_f(\alpha)
  \end{equation}
  for all $k\in\mathbb N$ and all $f \in E_\bbR$.  If $\lambda_0$ is a
  pole of order $k_0$, then $\lim_{\lambda \to \lambda_0}(\lambda -
  \lambda_0)^{k_0}R(\lambda,A)$ exists in $\mathcal L(E)$ and the limit
  is non-zero. However, the right hand side of \eqref{eq:pole-test}
  converges to $0$ as $\alpha \downarrow 0$ if $k>m$. Hence, $k_0\leq
  m$.
\end{proof}
It is another remarkable property of positive semigroups that on many
important spaces their spectral bound and their growth bound always
coincide (see \cite[Theorem~C-IV-1.1(a)]{Arendt1986}).  In the next
theorem, we show that this remains true for individually eventually
positive semigroups, for essentially the same reasons.
\begin{theorem}
  \label{thm:spb-growth-bd}
  Let $(e^{tA})_{t \geq 0}$ be an individually eventually positive
  $C_0$-semigroup on a complex Banach lattice $E$. Then $\spb(A) =
  \gbd(A)$ in any of the following cases:
  \begin{enumerate}[\upshape (i)]
  \item $E$ is a Hilbert space.
  \item $E = L^1(\Omega, \Sigma, \mu)$ for an arbitrary measure space
    $(\Omega, \Sigma, \mu)$ with $\mu \geq 0$.
  \item $E = C(K)$ for a compact Hausdorff space $K$ and $A$ is real.
  \end{enumerate}
\end{theorem}
\begin{proof}
  Note that it is sufficient to prove in each case that every
  individually eventually positive (and, in case (iii), real) semigroup
  on $E$ satisfies the implication
  \begin{equation}
    \label{form_implication_for_equality_of_spec_bnd_of_growth_bnd}
    \spb(A) < 0 \quad \Rightarrow \quad (e^{tA})_{t\geq 0} \text{ is
      bounded.}
  \end{equation}
  Indeed,
  \eqref{form_implication_for_equality_of_spec_bnd_of_growth_bnd} yields
  for every individually eventually positive (and, in case (iii), real)
  semigroup $(e^{tA})_{t \geq 0}$ that the rescaled semigroup
  $(e^{t(A-\alpha)})_{t \geq 0}$ is bounded whenever $\alpha > \spb(A)$;
  this in turn implies $\gbd(A) \leq \alpha$ and hence $\gbd(A) \le
  \spb(A)$.

  (i) Suppose that $\spb(A) < 0$. Then we have $s_0(A) < 0$ according to
  Corollary~\ref{cor:boundedness-of-resolvent}.  The Gearhart--Pr\"uss
  theorem, see \cite[Theorem~V.1.11]{Engel2000} or
  \cite[Theorem~5.2.1]{Arendt2011}, now implies that $e^{tA}$ converges
  to $0$ with respect to the operator norm as $t \to \infty$; in
  particular, $(e^{tA})_{t \ge 0}$ is bounded.

  (ii) Let $E = L^1(\Omega, \Sigma, \mu)$ and suppose that $\spb(A) <
  0$. Given $f \in E_+$, choose $t_0 \geq 0$ such that $e^{tA}f \geq 0$
  for all $t \geq t_0$. Proposition~\ref{prop:laplace-transform} yields
  that $R(0,A)f=\int_0^\infty e^{tA} f \, dt$ exists as an improper
  Riemann integral. Hence, due to the additivity of $\|\phdot\|_1$ on
  the positive cone $E_+$, we obtain
  \begin{displaymath}
    \int_0^\infty\|e^{tA}f\|_1\,dt
    \leq\Bigl\|\int_0^\infty e^{tA} f\,dt\Bigr\|_1
    + 2\int_0^{t_0} \|e^{tA}f\|_1 \,dt.
  \end{displaymath}
  Therefore, $\int_0^\infty \|e^{tA}f\|_1\, dt < \infty$ for every $f
  \in E_+$ and thus for every $f \in E$. By a theorem due to Datko and
  Pazy (see \cite[Theorem~5.1.2]{Arendt2011} or
  \cite[Theorem~V.1.8]{Engel2000}) we conclude that $(e^{tA})_{t\geq 0}$
  converges to $0$ with respect to the operator norm and is in
  particular bounded.

  (iii) Let $E=C(K)$ and suppose $A$ is real and $\spb(A) < 0$. Given
  $f \in E$, Proposition~\ref{prop:laplace-transform} implies that the
  Laplace transform of the trajectory
  \begin{displaymath}
    [0,\infty) \to E \text{,} \quad t \mapsto e^{tA}f
  \end{displaymath}
  has an abscissa of convergence which is no larger than $\spb(A)$ (see
  \cite[Start of Section~1.4]{Arendt2011} for a definition of the
  abscissa of convergence). Hence,
  \cite[Proposition~1.4.5(a)]{Arendt2011} implies that $\lim_{t \to
    \infty} \int_0^t e^{-\lambda s}e^{sA}\, ds$ exists in the operator
  norm whenever $\repart \lambda > \spb(A)$. Setting $\lambda = 0$ we
  obtain in particular that
  \begin{displaymath}
    M:=\sup_{t\geq 0}\Bigl\|\int_0^t e^{sA} \, ds\Bigr\|_\infty
    < \infty \text{.}
  \end{displaymath}
  Next, we choose an element $u \in D(A)$ with $u\gg 0$.  Such an
  element $u$ exists since $D(A)$ is dense in $C(K)$ and $A$ is real.
  Hence, for all $t>0$
  \begin{displaymath}
    \|e^{tA}u - u\|_\infty 
    = \Bigl\|\int_0^t e^{sA} Au \, ds \Bigr\|_\infty
    \leq M\|Au\|_\infty \text{,}
  \end{displaymath}
  so the trajectory $(e^{tA}u)_{t \geq 0}$ is bounded. Now let $f \in
  E_+$ with $\|f\|_\infty=1$. As $u\gg 0$ there exists $\beta>0$ such
  that $0\leq f\leq\one\leq\beta u$. Hence, by the eventual positivity
  of the semigroup there exists $t_0>0$ such that $0\leq e^{tA}f \leq
  \beta e^{tA} u$ for all $t\geq t_0$. This shows that the trajectory
  $(e^{tA}f)_{t \geq 0}$ is bounded for all $f \in E_+$ and hence for
  all $f \in E$. The uniform boundedness principle finally implies that
  $(e^{tA})_{t \geq 0}$ is bounded.
\end{proof}
For positive semigroups, the assertion of
Theorem~\ref{thm:spb-growth-bd} is also known to hold on $L^p(\Omega,
\Sigma, \mu)$ for $\sigma$-finite measure spaces $(\Omega, \Sigma, \mu)$
(see \cite[Theorem~5.3.6]{Arendt2011} or
\cite[Theorem~3.5.3]{Neerven1996}) and on $C_0(L)$-spaces for locally
compact Hausdorff spaces $L$ (see \cite[Theorem~B-IV-1.4]{Arendt1986}).
It would be interesting to know whether those results remain true for
individually or at least for uniformly eventually positive semigroups.

On $C(K)$ spaces, Theorem~\ref{thm:spb-growth-bd} yields the following
result on the non-emptiness of the spectrum of the generator.

\begin{corollary}
  Let $(e^{tA})_{t \geq 0}$ be a real and uniformly eventually strongly
  positive $C_0$-semigroup on $C(K)$ for some compact Hausdorff space
  $K$.  Then $\sigma(A)\neq\emptyset$.
\end{corollary}
\begin{proof}
  For sufficiently large $t_0 > 0$, the operator $e^{t_0A}$ is strongly
  positive. Thus, we have $e^{t_0A}\one \geq \varepsilon \one$ for some
  $\varepsilon > 0$. Iterating this inequality, we obtain that
  $(e^{t_0A})^n \one \geq \varepsilon^n \one$ for all $n \in \mathbb{N}$. 
  Hence, $\|(e^{t_0A})^n\|\geq \varepsilon^n$ and therefore 
  $\spr(e^{t_0A})\geq\varepsilon$. As $\spr(e^{t_0A}) = e^{t_0 \gbd(A)}$ (see
  \cite[Proposition~IV.2.2]{Engel2000}), we conclude that $\gbd(A)>
  -\infty$. Since the semigroup is real, so is $A$ and
  Theorem~\ref{thm:spb-growth-bd} thus implies $\gbd(A) =
  \spb(A)$. Hence, $\spb(A) > -\infty$.
\end{proof}

\begin{remark}
  For the generator $A$ of a positive semigroup on $C(K)$, the fact that
  $\sigma(A) \neq \emptyset$ is true without any irreducibility or
  strong positivity assumptions, see
  \cite[Theorem~B-III-1.1]{Arendt1986}. It does not seem clear whether
  $\sigma(A)\neq\emptyset$ for a uniformly or individually eventually
  positive semigroup on $C(K)$ in general.
\end{remark}

\section{Final remarks on eventually positive resolvents}
\label{sec:resolvents_final_remarks}

After discussing the spectral bound of individually eventually positive
semigroups, let us finish with a few notes on individually eventually
positive resolvents on arbitrary Banach lattices. If $E$ is a complex
Banach lattice, $A$ is a closed operator on $E$ and $\lambda_0$ is
either $-\infty$ or a spectral value of $A$ in $\bbR$ then, in complete
analogy to Section~\ref{sec:resolvents}, the resolvent of $A$ is called
\emph{individually eventually positive at $\lambda_0$} if there is a
number $\lambda_2 > \lambda_0$ with the following properties:
$(\lambda_0, \lambda_2] \subseteq \rho(A)$ and for each $f \in E_+$
there exists $\lambda_1 \in (\lambda_0, \lambda_2]$ such that
$R(\lambda,A)f \ge 0$ for all $\lambda \in (\lambda_0, \lambda_1]$. Let
us first make the following simple observation.

\begin{proposition}
  \label{prop:d_a_plus-is-generating}
  Let $A$ be a closed real operator on a complex Banach lattice $E$, let
  $\lambda_0$ be either $-\infty$ or a spectral value of $A$ in $\bbR$
  and suppose that the resolvent of $A$ is individually eventually
  positive at $\lambda_0$. Then the cone $D(A)_+ := D(A) \cap E_+$ is
  generating in $D(A)$, that is, $D(A) = D(A)_+ - D(A)_+$.
\end{proposition}
\begin{proof}
  Let $f \in D(A)$ and choose $\lambda > \lambda_0$ sufficiently small
  such that $\lambda \in \rho(A)$ and such that $R(\lambda,A)f^+ \ge 0$,
  $R(\lambda,A)f^- \ge 0$, $R(\lambda,A)(Af)^+ \ge 0$ and
  $R(\lambda,A)(Af)^- \ge 0$. We then have
  \begin{displaymath}
    \begin{split}
      f &= R(\lambda,A)(\lambda - A)f \\
      &=\lambda R(\lambda,A)f^+ - \lambda R(\lambda,A)f^- -
      R(\lambda,A)(Af)^+ + R(\lambda,A)(Af)^- \text{,}
    \end{split}
  \end{displaymath}
  which is clearly contained in $D(A)_+ - D(A)_+$.
\end{proof}

One might ask whether $D(A) \cap E_+$ is also generating in $D(A)$ if
$A$ is the generator of an individually eventually positive semigroup,
and one might also wonder whether the resolvent of the generator $A$ of
an individually eventually positive semigroup is always individually
eventually positive at $\spb(A)$. The following example shows that the
answer to both questions is negative, even if the semigroup is assumed
to be uniformly eventually positive.

\begin{example}
  There is a real $C_0$-semigroup $(e^{tA})_{t \ge 0}$ on a complex Banach
  lattice $E$ with the following properties:
  \begin{enumerate}[(a)]
  \item The semigroup $(e^{tA})_{t \ge 0}$ is nilpotent and therefore
    uniformly eventually positive.
  \item The resolvent $R(\phdot,A)$ is not individually eventually
    positive at $\spb(A) = -\infty$.
  \item $D(A)_+ := D(A) \cap E_+$ is not generating in $D(A)$.
  \end{enumerate}
  Indeed, let $p \in [1,\infty)$ and let $E = L^p((0,1)) \oplus
  L^p((0,1))$.  We define a ``sign-flipping left shift semigroup''
  $(e^{tA})_{t \ge 0}$ on $E$ in the following way: For $(f_1,f_2) \in
  E$ we set $e^{tA}(f_1,f_2) = \bigl(g_1(t),g_2(t)\bigr)$, where
  \begin{displaymath}
    g_1(x,t) = 
    \begin{cases}
      f_1(x+t) \quad & \text{if $0\leq x+t \le 1$,} \\
      -f_2(x-1+t) \quad & \text{if $1 < x+t \le 2$,} \\
      0 \quad & \text{if $2 < x+t$,}
    \end{cases}
  \end{displaymath}
  and
  \begin{displaymath}
    g_2(x,t) = 
    \begin{cases}
      f_2(x+t) \quad & \text{if $0\leq x+t \le 1$,} \\
      0 \quad & \text{if $1<x+t$.}
    \end{cases}
  \end{displaymath}
  Clearly this semigroup is nilpotent. Similarly as in
  \cite[Section~A-I-2.6]{Arendt1986}, it can be proved that its 
  generator $A$ is given by
  \begin{displaymath}
    \begin{aligned}
      D(A) & =\bigl\{(f_1,f_2)\in W^{1,p}((0,1))\oplus W^{1,p}((0,1))
      \colon\text{$f_1(1) = - f_2(0)$ and $f_2(1)=0$}\bigr\} \text{,} \\
      A(f_1,f_2) & = (f_1',f_2') \text{.}
    \end{aligned}
  \end{displaymath}
  In particular we have $f_1(1) = f_2(1) = 0$ for each tuple $(f_1,f_2)
  \in D(A)_+$.  Hence, $D(A)_+$ is not generating in $D(A)$. By
  Proposition~\ref{prop:d_a_plus-is-generating} this implies that the
  resolvent $R(\phdot,A)$ is not individually eventually positive at
  $-\infty$.
\end{example}

\paragraph{Acknowledgements}
The authors would like to express their warmest thanks to Wolfgang
Arendt for many enlightening discussions and for contributing several
ideas and proofs.  They would also like to thank Anna Dall'Acqua for
pointing out several references, and the referee for a careful and
thoughtful reading of the manuscript. This paper was initiated during a
very pleasant visit of the first author to Ulm University.

\pdfbookmark[1]{\refname}{refs}
%\bibliography{literature_evpos}

\begin{thebibliography}{10}

\bibitem{amann:83:dss}
H.~Amann, \emph{Dual semigroups and second order linear elliptic boundary value
  problems}, Israel J. Math. \textbf{45} (1983), 225--254.
  DOI:\,\href{http://dx.doi.org/10.1007/BF02774019}{\nolinkurl{10.1007/BF02774019}}

\bibitem{MR933321}
W.~Arendt and C.~J.~K. Batty, \emph{Tauberian theorems and stability of
  one-parameter semigroups}, Trans. Amer. Math. Soc. \textbf{306} (1988),
  837--852.
  DOI:\,\href{http://dx.doi.org/10.2307/2000826}{\nolinkurl{10.2307/2000826}}

\bibitem{Arendt1986}
W.~Arendt, A.~Grabosch, G.~Greiner, U.~Groh, H.~P. Lotz, U.~Moustakas,
  R.~Nagel, F.~Neubrander, and U.~Schlotterbeck, \emph{One-parameter semigroups
  of positive operators}, Lecture Notes in Mathematics, vol. 1184,
  Springer-Verlag, Berlin, 1986.
  DOI:\,\href{http://dx.doi.org/10.1007/BFb0074922}{\nolinkurl{10.1007/BFb0074922}}

\bibitem{MR3146835}
W.~Arendt, A.~F.~M. ter Elst, J.~B. Kennedy, and M.~Sauter, \emph{The
  {D}irichlet-to-{N}eumann operator via hidden compactness}, J. Funct. Anal.
  \textbf{266} (2014), 1757--1786.
  DOI:\,\href{http://dx.doi.org/10.1016/j.jfa.2013.09.012}{\nolinkurl{10.1016/j.jfa.2013.09.012}}

\bibitem{Arendt2011}
W.~Arendt, C.~J.~K. Batty, M.~Hieber, and F.~Neubrander, \emph{Vector-valued
  {L}aplace transforms and {C}auchy problems}, second ed., Monographs in
  Mathematics, vol.~96, Birkh\"auser/Springer Basel AG, Basel, 2011.
  DOI:\,\href{http://dx.doi.org/10.1007/978-3-0348-0087-7}{\nolinkurl{10.1007/978-3-0348-0087-7}}

\bibitem{ABR90}
W.~Arendt, C.~J.~K. Batty, and D.~W. Robinson, \emph{Positive semigroups
  generated by elliptic operators on {L}ie groups}, J. Operator Theory
  \textbf{23} (1990), 369--407.

\bibitem{arendt:12:fei}
W.~Arendt and R.~Mazzeo, \emph{Friedlander's eigenvalue inequalities and the
  {D}irichlet-to-{N}eumann semigroup}, Commun. Pure Appl. Anal. \textbf{11}
  (2012), 2201--2212.
  DOI:\,\href{http://dx.doi.org/10.3934/cpaa.2012.11.2201}{\nolinkurl{10.3934/cpaa.2012.11.2201}}

\bibitem{Brauer1961}
A.~Brauer, \emph{On the characteristic roots of power-positive matrices}, Duke
  Math. J. \textbf{28} (1961), 439--445.
  DOI:\,\href{http://dx.doi.org/10.1215/S0012-7094-61-02840-X}{\nolinkurl{10.1215/S0012-7094-61-02840-X}}

\bibitem{campbell:13:lac}
A.~P. Campbell and D.~Daners, \emph{Linear {A}lgebra via {C}omplex {A}nalysis},
  Amer. Math. Monthly \textbf{120} (2013), 877--892.
  DOI:\,\href{http://dx.doi.org/10.4169/amer.math.monthly.120.10.877}{\nolinkurl{10.4169/amer.math.monthly.120.10.877}}

\bibitem{AcSw05}
A.~Dall'Acqua and G.~Sweers, \emph{The clamped-plate equation for the lima\c
  con}, Ann. Mat. Pura Appl. (4) \textbf{184} (2005), 361--374.
  DOI:\,\href{http://dx.doi.org/10.1007/s10231-004-0121-9}{\nolinkurl{10.1007/s10231-004-0121-9}}

\bibitem{DanersII}
D.~Daners, J.~Gl\"uck, and J.~Kennedy, \emph{Eventually and asymptotically
  positive semigroups on {B}anach lattices}, In preparation.

\bibitem{Daners2014}
D.~Daners, \emph{Non-positivity of the semigroup generated by the
  {D}irichlet-to-{N}eumann operator}, Positivity \textbf{18} (2014), 235--256.
  DOI:\,\href{http://dx.doi.org/10.1007/s11117-013-0243-7}{\nolinkurl{10.1007/s11117-013-0243-7}}

\bibitem{Engel2000}
K.-J. Engel and R.~Nagel, \emph{One-parameter semigroups for linear evolution
  equations}, Graduate Texts in Mathematics, vol. 194, Springer-Verlag, New
  York, 2000, With contributions by S. Brendle, M. Campiti, T. Hahn, G.
  Metafune, G. Nickel, D. Pallara, C. Perazzoli, A. Rhandi, S. Romanelli and R.
  Schnaubelt.
  DOI:\,\href{http://dx.doi.org/10.1007/b97696}{\nolinkurl{10.1007/b97696}}

\bibitem{MR1288366}
J.~Escher, \emph{The {D}irichlet-{N}eumann operator on continuous functions},
  Ann. Scuola Norm. Sup. Pisa Cl. Sci. (4) \textbf{21} (1994), 235--266.
  Available at
  \href{http://www.numdam.org/item?id=ASNSP_1994_4_21_2_235_0}{\nolinkurl{http://www.numdam.org/item?id=ASNSP_1994_4_21_2_235_0}}

\bibitem{FGG08}
A.~Ferrero, F.~Gazzola, and H.-C. Grunau, \emph{Decay and eventual local
  positivity for biharmonic parabolic equations}, Discrete Contin. Dyn. Syst.
  \textbf{21} (2008), 1129--1157.
  DOI:\,\href{http://dx.doi.org/10.3934/dcds.2008.21.1129}{\nolinkurl{10.3934/dcds.2008.21.1129}}

\bibitem{Furstenberg1981}
H.~Furstenberg, \emph{Recurrence in ergodic theory and combinatorial number
  theory}, Princeton University Press, Princeton, N.J., 1981, M. B. Porter
  Lectures.

\bibitem{GaGr08}
F.~Gazzola and H.-C. Grunau, \emph{Eventual local positivity for a biharmonic
  heat equation in {$\mathbb R^n$}}, Discrete Contin. Dyn. Syst. Ser. S
  \textbf{1} (2008), 83--87.
  DOI:\,\href{http://dx.doi.org/10.3934/dcdss.2008.1.83}{\nolinkurl{10.3934/dcdss.2008.1.83}}

\bibitem{Grobler1995}
J.~J. Grobler, \emph{Spectral theory in {B}anach lattices}, Operator theory in
  function spaces and {B}anach lattices, Oper. Theory Adv. Appl., vol.~75,
  Birkh\"auser, Basel, 1995, pp.~133--172.
  DOI:\,\href{http://dx.doi.org/10.1007/978-3-0348-9076-2_10}{\nolinkurl{10.1007/978-3-0348-9076-2_10}}

\bibitem{GrSw96}
H.-C. Grunau and G.~Sweers, \emph{Positivity for perturbations of polyharmonic
  operators with {D}irichlet boundary conditions in two dimensions}, Math.
  Nachr. \textbf{179} (1996), 89--102.
  DOI:\,\href{http://dx.doi.org/10.1002/mana.19961790106}{\nolinkurl{10.1002/mana.19961790106}}

\bibitem{MR832183}
R.~A. Horn and C.~R. Johnson, \emph{Matrix analysis}, Cambridge University
  Press, Cambridge, 1985.
  DOI:\,\href{http://dx.doi.org/10.1017/CBO9780511810817}{\nolinkurl{10.1017/CBO9780511810817}}

\bibitem{MR2117663}
C.~R. Johnson and P.~Tarazaga, \emph{On matrices with {P}erron-{F}robenius
  properties and some negative entries}, Positivity \textbf{8} (2004),
  327--338.
  DOI:\,\href{http://dx.doi.org/10.1007/s11117-003-3881-3}{\nolinkurl{10.1007/s11117-003-3881-3}}

\bibitem{Kato1976}
T.~Kato, \emph{Perturbation theory for linear operators}, second ed.,
  Springer-Verlag, Berlin, 1976, Grundlehren der Mathematischen Wissenschaften,
  Band 132.
  DOI:\,\href{http://dx.doi.org/10.1007/978-3-642-66282-9}{\nolinkurl{10.1007/978-3-642-66282-9}}

\bibitem{katznelson:04:iha}
Y.~Katznelson, \emph{An introduction to harmonic analysis}, third ed.,
  Cambridge Mathematical Library, Cambridge University Press, Cambridge, 2004.

\bibitem{MiOk86}
S.~Miyajima and N.~Okazawa, \emph{Generators of positive {$C_0$}-semigroups},
  Pacific J. Math. \textbf{125} (1986), 161--176. Available at
  \href{http://projecteuclid.org/euclid.pjm/1102700217}{\nolinkurl{http://projecteuclid.org/euclid.pjm/1102700217}}

\bibitem{Noutsos2006}
D.~Noutsos, \emph{On {P}erron-{F}robenius property of matrices having some
  negative entries}, Linear Algebra Appl. \textbf{412} (2006), 132--153.
  DOI:\,\href{http://dx.doi.org/10.1016/j.laa.2005.06.037}{\nolinkurl{10.1016/j.laa.2005.06.037}}

\bibitem{Noutsos2008}
D.~Noutsos and M.~J. Tsatsomeros, \emph{Reachability and holdability of
  nonnegative states}, SIAM J. Matrix Anal. Appl. \textbf{30} (2008), 700--712.
  DOI:\,\href{http://dx.doi.org/10.1137/070693850}{\nolinkurl{10.1137/070693850}}

\bibitem{Schaefer1974}
H.~H. Schaefer, \emph{Banach lattices and positive operators}, Springer-Verlag,
  New York, 1974, Die Grundlehren der mathematischen Wissenschaften, Band 215.
  DOI:\,\href{http://dx.doi.org/10.1007/978-3-642-65970-6}{\nolinkurl{10.1007/978-3-642-65970-6}}

\bibitem{Seneta1981}
E.~Seneta, \emph{Nonnegative matrices and {M}arkov chains}, second ed.,
  Springer Series in Statistics, Springer-Verlag, New York, 1981.
  DOI:\,\href{http://dx.doi.org/10.1007/0-387-32792-4}{\nolinkurl{10.1007/0-387-32792-4}}

\bibitem{Neerven1996}
J.~van Neerven, \emph{The asymptotic behaviour of semigroups of linear
  operators}, Operator Theory: Advances and Applications, vol.~88, Birkh\"auser
  Verlag, Basel, 1996.
  DOI:\,\href{http://dx.doi.org/10.1007/978-3-0348-9206-3}{\nolinkurl{10.1007/978-3-0348-9206-3}}

\bibitem{Warma2006}
M.~Warma, \emph{The {R}obin and {W}entzell-{R}obin {L}aplacians on {L}ipschitz
  domains}, Semigroup Forum \textbf{73} (2006), 10--30.
  DOI:\,\href{http://dx.doi.org/10.1007/s00233-006-0617-2}{\nolinkurl{10.1007/s00233-006-0617-2}}

\bibitem{Yosida1995}
K.~Yosida, \emph{Functional analysis}, Classics in Mathematics,
  Springer-Verlag, Berlin, 1995.
  DOI:\,\href{http://dx.doi.org/10.1007/978-3-642-61859-8}{\nolinkurl{10.1007/978-3-642-61859-8}}

\end{thebibliography}

\providecommand{\bysame}{\leavevmode\hbox to3em{\hrulefill}\thinspace}

\end{document}